\documentclass[a4paper,12pt]{article}
\usepackage[utf8x]{inputenc}
\usepackage[T1]{fontenc}
\usepackage[a4paper,includeheadfoot,margin=2.54cm]{geometry}
\usepackage{stmaryrd}
\usepackage{amsthm}%theoremstyle
\usepackage{mathtools}% an extension package to amsmath
\mathtoolsset{showonlyrefs}
\usepackage[all]{xy}
\usepackage{xcolor}
\usepackage{enumitem}
\setenumerate[1]{itemsep=0pt,partopsep=0pt,parsep=\parskip,topsep=5pt}
\setitemize[1]{itemsep=0pt,partopsep=0pt,parsep=\parskip,topsep=5pt}
\setdescription{itemsep=0pt,partopsep=0pt,parsep=\parskip,topsep=5pt}
\usepackage{times}%times as default text fonts
\usepackage[scaled=.92]{helvet}
\usepackage{mathabx}% fonts for symbols,load before amsrefs
\usepackage[cal=esstix,scr=rsfs,bb=ams,frak=euler]{mathalfa}
\usepackage{amsrefs}
\usepackage[pdftex]{graphicx}
\usepackage[pdftex,
colorlinks=true,
linkcolor=purple,
citecolor=blue,
hyperindex,
plainpages=false,
bookmarks=true,
bookmarksopen=true,
bookmarksnumbered]{hyperref}
\theoremstyle{plain}
\newtheorem{theorem}{Theorem}[section]
\newtheorem*{theorem*}{Theorem}
\newtheorem{corollary}[theorem]{Corollary}
\newtheorem{lemma}[theorem]{Lemma}
\newtheorem{proposition}[theorem]{Proposition}

\theoremstyle{definition}

\theoremstyle{remark}
\newtheorem{remark}[theorem]{Remark}

\numberwithin{equation}{section}

\newcommand{\sd}{\mathbin{\triangle}}%symmetric difference
\newcommand{\grass}[2]{\mathbf{G}(#1,#2)}% Grassmannian "G(n,m)" with bold "G"
\newcommand{\orthg}{\mathbf{O}}% orthogonal group "O(n)" with bold "O"
\newcommand{\oball}{\mathbf{U}}% open ball "U(a,r)" with bold "U"
\newcommand{\cball}{\mathbf{B}}% closed ball "B(a,r)" with bold "B"
\newcommand{\HM}{\mathcal{H}}% the Hausdorff measure
\newcommand{\HD}{{d_{\HM}}}% the Hausdorff distance
\newcommand{\size}{\mathsf{Size}} 
\newcommand{\mass}{\mathsf{M}} 
\newcommand{\fn}{\mathsf{W}} 
\newcommand{\RC}{\mathscr{R}} 
\newcommand{\IC}{\mathbf{I}} 
\newcommand{\NC}{\mathscr{N}} 
\newcommand{\PC}{\mathscr{P}} 
\newcommand{\FC}{\mathscr{F}} 
\newcommand{\LC}{\mathscr{L}} 
\DeclareMathOperator{\spt}{spt}% support of measures
\DeclareMathOperator{\id}{id}%identity mapping
\DeclareMathOperator{\Tan}{Tan}% tangent space
\DeclareMathOperator{\Lip}{Lip}% Lipschitz constant
\DeclareMathOperator{\dist}{dist}% distance
\DeclareMathOperator{\diam}{diam}% diameter
\DeclareMathOperator{\ap}{ap} % something approximate
\DeclareMathOperator{\intr}{int} % interior of a set
\newcommand{\scale}[1]{\boldsymbol{\mu}_{#1}}
\newcommand{\trans}[1]{\boldsymbol{\tau}_{#1}}
\newcommand{\gm}[2]{\boldsymbol{\gamma}_{#1,#2}}
\newcommand{\mr}{\mathop{\vrule height 1.6ex depth 0pt width
0.13ex\vrule height 0.13ex depth 0pt width 1.3ex}\nolimits}
\newcommand{\Hom}{\mathsf{H}}% Homology groups
\newcommand{\cech}{\check{\mathsf{H}}}% Cech homology groups
\newcommand{\cc}{\check{\mathscr{C}}}% Cech spans
\newcommand{\ch}{\mathscr{C}_{\mathsf{H}}}% homology spans
\newcommand{\ora}[1]{{\scriptstyle\xrightarrow{#1}}}%something over rightarrow 
\begin{document}
\title{On the equality between the infimum obtained by solving various Plateau's
problem}
\author{Yangqin Fang \and Vincent Feuvrier \and Chunyan Liu}

\renewcommand{\thefootnote}{\fnsymbol{footnote}} 
\renewcommand{\thefootnote}{\arabic{footnote}} 

\date{}
\maketitle
\begin{abstract}
	In this paper we will compare the Plateau's problem with \v{C}ech and
	singular homological boundary conditions, we also compare these with the size
	minimizing problem for integral currents with a given boundary. Finally we
	get the agreement on the infimum values for these Plateau's problem.
\end{abstract}

\section{Introduction and notation} 
Plateau's problem is to show the existence of minimal {\em surfaces} with a
given boundary. The classical form of the problem was solved independently by
J. Douglas \cite{Douglas:1931} and T. Rad\'o \cite{Rado:1930}, in which the
surfaces are understood as parametrizations of the unit disc in
$\mathbb{R}^3$, and such surfaces are analytic which is proved by Osserman
\cite{Osserman:1970}. More general settings for the problem have been
developed by many authors in various ways
\cite{FF:1960,Reifenberg:1960,David:2012}, in which the solutions are more
like models of soap films. The variation arises from the different
understanding  of surfaces. 

Integral currents, introduced by Federer and    Fleming \cite{FF:1960}, are
considered as an excellent generalization of surfaces, since they possess a
canonical boundary operator, and soap films could be considered as size
minimal integral currents. The existence of mass minimal (integral or
rectifiable) currents follows from a compactness theorem, see  \cite{FF:1960}.
However, the existence of size minimal integral or rectifiable currents with
given boundary is still open. A special case was solved by Morgan
\cite{Morgan:1989}. 

Surfaces could be directly understood as point sets, but 
one difficulty arise about how to understand the boundary. Reifenberg
\cite{Reifenberg:1960} proposed Plateau's problem with {\v C}ech homological 
boundary conditions. That is, given a  compact set $B$ and an abelian group 
$G$, a compact set $E\supseteq B$ is called a set with algebraic
boundary containing $L\subseteq \cech_{d-1}(B;G)$ (or called spanning $L$) if 
$\cech_{d-1} (i_{B,E})(L)=0$, where we denote by $\cech_{d-1}(i_{B,E})$ the 
natural homomorphism induced by the inclusion mapping $B\to E$, and in this
paper we always assume that $d$ is a positive integer which does not exceed 
the dimension of ambient euclidean space $\mathbb{R}^n$. Then we consider the 
problem of minimizing the Hausdorff measure $\HM^d(E)$ in the class consist of
all compact sets whose algebraic boundary contains $L$.  
Reifenberg \cite{Reifenberg:1960} got the first existence result for the
minimizing problem under the assumption
that $B$ is a $(d-1)$-dimensional compact set and $G$ is a compact abelian group.
Under a more general setting, Almgren \cite{Almgren:1968} claimed (with 
vague proof) the existence with Vietoris homology and any abelian group $G$ 
to minimize the integral of an elliptic integrand. And recently, a proof in
full detail was presented in \cite{FS:2018}, which indeed inherit Almgren's ideas.
De Pauw \cite{Pauw:2009} solved the minimization problem in the case that $B$ is a smooth Jordan curve in
$\mathbb{R}^3$ with $G=\mathbb{Z}$ in a totally different way. We obtained in
\cite{Fang:2013} a general existence result with a wild class of integrands 
which also implies the existence of solutions to the Plateau Problem with 
\v{C}ech homology conditions, and the proof is based on lots of work of the second
author \cite{Feuvrier:2008}. We can also replace the {\v C}ech homology with
other homology $\Hom_{\ast}$ in the minimization problem, in fact there are 
abundant existence results to Plateau's
problem with \v Cech homology, but very little with other homology. For
convenience, we call all such kinds of minimization problems the Plateau's problems
with homological boundary conditions. 

In this paper, we will show that the infimum values for the Plateau's problem
with the homological boundary conditions coincide if the homologies satisfy 
the seven axioms of Eilenberg and Steenrod \cite{ES:1945,ES:1952} and the 
additivity axiom \cite{Milnor:1962}, and that 
value also coincides with the infimum value in the corresponding size minimizing
problem for integral currents when $G=\mathbb{Z}$, see Theorem \ref{thm:hequ}
and Corollary \ref{co:scequ} for details. Let us introduce some notation. A 
real-valued function on a metric space is called of class Baire 1, if it is a 
pointwise limit of a sequence of continuous functions. For any set $E\subseteq
\mathbb{R}^n$, the $d$-dimensional Hausdorff measure $\HM^d(E)$ is defined by  
\[
	\HM^d(E)=\lim_{\delta\to 0}\inf\left\{\sum \diam(A_i)^d: E\subseteq
	\bigcup_{i=1}^{\infty} A_i, A_i\subseteq \mathbb{R}^n,\diam(A_i)\leq \delta\right\}.
\]
For any $d$-plane $T$ in $\mathbb{R}^n$, we denote by $T_{\natural}$ the
orthogonal projection of $\mathbb{R}^n$ onto $T$, and by $T_{\natural}^{\perp}$ 
the mapping $\id-T_{\natural}$. Let $\grass{n}{d}$ be the Grassmannian manifold, 
which consists of $d$-dimensional subspace in $\mathbb{R}^{n}$ equipped with 
the metric $d(T,P)=\|T_{\natural}-P_{\natural}\|$ for any $T,P\in \grass{n}{d}$. 
We denote by $\gm{n}{d}$ the $\orthg(n)$-invariant probability measure on
$\grass{n}{d}$. An integrand 
is a $\HM^d$-measurable function $F:\mathbb{R}^n\times \grass{n}{d}\to 
(0,+\infty)$, and we call it bounded above if $ \sup F<+\infty$.
For any $\HM^d$-measurable set $E\subseteq \mathbb{R}^n$ with $\HM^d(E)<+\infty$, we
can write $E=E_{irr} \sqcup E_{rec}$, where $E_{rec}$ and $E_{irr}$ are 
rectifiable and purely unrectifiable respectively, and both of them are 
$\HM^d$-measurable. For any set function 
$\lambda:2^{\mathbb{R}^n}\to \mathbb{R}$, $\lambda$ is called positive if
$\lambda(E)\geq 0$ for any $E\subseteq \mathbb{R}^n$, $\lambda$ is called
bounded above if $\lambda(E)\leq b<\infty$ for some $b\in (0,+\infty)$ and
all $E\subseteq \mathbb{R}^n$. Given any integrand $F$, set function
$\lambda$ and $\HM^d$-measurable $d$-set $E\subseteq \mathbb{R}^n$, we define 
$\Phi_{F,\lambda}(E)$ by
\[
	\Phi_{F,\lambda}(E)=\int_{x\in
	E_{rec}}F(x,\Tan(E,x))d\HM^d(x)+\lambda(E_{irr})\HM^d(E_{irr}).
\]
In particular, if $F\equiv 1$ and $\lambda \equiv 1$, then $\Phi_{F,\lambda}
=\HM^d$. Let $G$ be an abelian group, $\Hom_{\ast}$ be a homology. Let
$B\subseteq \mathbb{R}^n$ be a compact set, and $L\subseteq \Hom_{d-1}(B;G)$
be a subgroup. Let $f:\mathbb{R}^n\to \mathbb{R}^n$ be any continuous mapping,
we denote by $\Hom_{\ast}(f)$ the homomorphism induced by $f$. We denote by
$\mathscr{C}_{\Hom}(B,G,L)$ the collection of all  compact subsets in
$\mathbb{R}^{n}$ which span $L$ in homology $\Hom_{\ast}$,  for the \v{C}ech
homology, we will simply denote it by $\cc(B,G,L)$.  A set $X\subseteq
\mathbb{R}^n$ is called a Lipschitz neighborhood retract if there  is an open
set $U\supseteq X$ and a Lipschitz mapping $\varphi:U\to X$ such that
$\varphi\vert_{X}=\id_X$; furthermore, if $X$ is compact, we will call it a
compact Lipschitz neighborhood retract. Then we have the following theorem:
\begin{theorem} \label{thm:hequ}
	Let $B_0\subseteq \mathbb{R}^n$ be a compact Lipschitz neighborhood retract,
	and $\rho:U\to B_0$ be a Lipschitz retraction. 
	Let $G$ be any abelian group, and $\Hom_{\ast}$ be any homology 
	satisfying that 
	\begin{itemize}[topsep=0pt,noitemsep]
		\item $\Hom_{d-1}(\id)=\id$;
		\item $\Hom_{d-1}(g\circ f)=\Hom_{d-1}(g)\circ
			\Hom_{d-1}(f)$ for any Lipschitz mappings
			$f:X\to Y$ and $g:Y\to Z$;
		\item $\Hom_{d-1}(f)=\Hom_{d-1}(g)$ for any two homotopic Lipschitz 
			mappings $f,g:\mathbb{R}^n\to \mathbb{R}^n$;
		\item there is an isomorphism $h_0:\Hom_{d-1}(B_0;G)
			\to \cech_{d-1}(B_0;G)$, and for any polyhedron $E\supseteq B_0$ there
			is an isomorphism $h_{E}:\Hom_{d-1}(E;G)
			\to \cech_{d-1}(E;G)$ such that $h_E\circ \Hom_{d-1}(i_{B_0,E})=\cech_{d-1}
			(i_{B_0,E})\circ h_0$, and if $E\subseteq U$ then $h_0\circ\Hom_{d-1}
			(\rho\vert_E)= \cech_{d-1}(\rho\vert_E)\circ h_E$.
	\end{itemize}
	Let $F$ be any integrand of class Baire 1 bounded above, and $\lambda$ be any
	positive set function bounded above. Suppose that $L$ is a subgroup of 
	$\Hom_{d-1}(B_0;G)$, $L' =  h_0(L)$. Then we have that
	\begin{equation}\label{eq:hequ00}
		\inf\{\Phi_{F,\lambda}(E\setminus B_0):E\in \ch(B_0,G,L)\}=
		\inf\{\Phi_{F,\lambda}(E\setminus B_0):E\in \cc(B_0,G,L')\}.
	\end{equation}
\end{theorem}
In fact, we do not assume the homology in the above theorem satisfies all of 
the seven axioms of Eilenberg and Steenrod; but evidently, if it satisfies the
seven axioms and the additivity axiom \cite{Milnor:1962}, then the above four 
conditions hold. Since both {\v C}ech and singular homology satisfy all these
axioms, see for example \cite{ES:1945,ES:1952,Milnor:1962},
we get particularly that the above theorem holds for singular homology.  

Let $G$ be any complete normed abelian group. Let $\PC_d(\mathbb{R}^n;G)$
be the group of polyhedral chains of dimension $d$ in $\mathbb{R}^n$. For
any $P\in \PC_d(\mathbb{R}^n;G)$, let $\mass(P)$ and
$\size(P)$ be the mass and the size of $P$ defined by
\[
	\mass(P)=\inf\left\{\sum |g_i|\HM^d(\sigma_i):P=\sum g_i \sigma_i\right\},\ 
	\size(P)=\inf\left\{\sum \HM^d(\sigma_i):P=\sum g_i \sigma_i\right\}.
\]
We put
\[
	\fn(P)=\inf\{\mass(Q)+\mass(R):P=Q+\partial R, Q\in \PC_d(\mathbb{R}^n;G),
	R\in \PC_{d+1}(\mathbb{R}^n;G)\}.
\]
Then $\fn$ is a norm on $\mathscr{P}_d(\mathbb{R}^n;G)$, see Theorem 2.2 in \cite
{Fleming:1966}. Let $\FC_d(\mathbb{R}^n;G)$ be the $\fn$-completion of
$\PC_d(\mathbb{R}^n;G)$. The elements of $\FC_d(\mathbb{R}^n;G)$ are called 
flat $G$-chains \cite{Fleming:1966}. For any $T\in \FC_d(\mathbb{R}^n;G)$, the
mass of $T$ is defined by 
\[
	\mass(T)=\inf\left\{\liminf_{k\to \infty}\mass(P_k):P_k\ora{\fn}T, P_k\in
	\PC_d(\mathbb{R}^n;G)\right\},
\]
and the flat size of $T$ is defined by 
\[
	\size(T)=\inf\left\{\liminf_{m\to\infty}\size(P_m): P_m\ora{\fn} T, P_m\in
	\mathscr{P}_{d}(\mathbb{R}^n;G)\right\}.
\]
For any Lipschitz mapping $f:\mathbb{R}^n\to \mathbb{R}^n$, there is an induced
homomorphism $f_{\sharp}:\FC_d(\mathbb{R}^n;G)\to \FC_d(\mathbb{R}^n;G)$, see
Section 5 in \cite{Fleming:1966}.
For any $A\in \FC_{d}(\mathbb{R}^n;G)$, $A$ is called rectifiable, if for any
$\varepsilon>0$ there is a Lipschitz mapping $f:\mathbb{R}^n\to \mathbb{R}^n$
and a polyhedral chain $P\in \PC_d(\mathbb{R}^n;G)$ such that 
$\mass(A-f_{\sharp}P)<\varepsilon$, we denote
by $\RC_{d}(\mathbb{R}^n;G)$ the collection of all such chains, define the group of Lipschitz chains and
normal chains by 
\[
	\LC_{d}(\mathbb{R}^n;G)=\{f_{\sharp}P:P\in \PC_d(\mathbb{R}^n;G)\}
\]
and
\[
	\NC_{d}(\mathbb{R}^n;G)=\{A\in
	\FC_d(\mathbb{R}^n;G):\mathsf{N}(A)<\infty\},
\]
respectively, where $\mathsf{N}(A)=\mass(A)+\mass(\partial A)$.
Then we see that $\RC_{d}(\mathbb{R}^n;G)$ is the $\mass$-completion of
$\LC_{d}(\mathbb{R}^n;G)$, $\FC_{d} (\mathbb{R}^n;G)$ is the $\fn$-completion 
of $\LC_{d}(\mathbb{R}^n;G)$, and we will see that $\NC_{d}(\mathbb{R}^n;G)$ is
the $\mathsf{N}$-completion of $\LC_{d}(\mathbb{R}^n;G)$. 
For any $\mathscr{S}\in \{\NC,\RC,\FC\}$, we put 
\[
	\mathscr{S}_d^c(\mathbb{R}^n;G)=\{A\in \mathscr{S}_d(\mathbb{R}^n;G):
	\spt(A) \text{ is compact}\}.
\]

For any set $X\subseteq \mathbb{R}^n$ and $\mathscr{S}\in \{\LC,\NC,\RC,\FC,
\NC^c,\RC^c,\FC^c\}$, we put 
\[
	\mathscr{S}_d(X;G)=\{A\in \mathscr{S}_{d}(\mathbb{R}^n;G): \spt A\subseteq
	X\}.
\]
\begin{equation}\label{eq:hg0}
	\begin{gathered}
		\mathcal{Z}_{d-1}^{\mathscr{S}}(X,B;G)=\{T\in
		\mathscr{S}_{d-1}(\mathbb{R}^n;G):\spt T\subseteq X,
		\spt \partial T\subseteq B \text{ or }d=1\},\\
		\mathcal{B}_{d-1}^{\mathscr{S}}(X,B;G)=\{T+\partial S: T\in
		\mathscr{S}_{d-1}(\mathbb{R}^n;G),
		\spt T\subseteq B, S\in \mathscr{S}_{d}(\mathbb{R}^n;G), \spt S\subseteq X \},
	\end{gathered}
\end{equation}
and define the homology group
\[
	\Hom_{d-1}^{\mathscr{S}}(X,B;G)=	\mathcal{Z}_{d-1}^{\mathscr{S}}(X,B;G)/
	\mathcal{B}_{d-1}^{\mathscr{S}}(X,B;G);
\]
if $B=\emptyset$, then we simply denote $\mathcal{Z}_{d-1}^{\mathscr{S}}(X,
\emptyset;G) $, $\mathcal{B}_{d-1}^{\mathscr{S}}(X,\emptyset;G)$ and 
$\Hom_{d-1}^{\mathscr{S}}(X,\emptyset;G)$
by $\mathcal{Z}_{d-1}^{\mathscr{S}}(X;G)$, $\mathcal{B}_{d-1}^{\mathscr{S}}
(X;G)$ and $\Hom_{d-1}^{\mathscr{S}}(X;G)$ respectively.

If $B_0\subseteq \mathbb{R}^n$ is a compact Lipschitz neighborhood retract, then
for any $\mathscr{S}\in \{\LC,\NC,\RC,\FC\}$, there is an isomorphism 
$\Psi_{d-1}^{\mathscr{S}}:\Hom_{d-1}^{\mathscr{S}}(B_0;G)\to \cech_{d-1}(B_0;G)$, 
see Proposition \ref{prop:hfc} for details. If $L\subseteq
\Hom_{d-1}^{\mathscr{S}}(B_0;G)$ is a subgroup, for simplifying notation, we use
$\cc(B_0,G,L)$ to denote the collection $\cc(B_0,G,\Psi_{d-1}^{\mathscr{S}}(L))$.
For any $T\in \mathcal{Z}_{d -1}(B_0;G)$, if we denote by $[T]=T+
\mathcal{B}_{d-1}^{\mathscr{S}}(B_0;G)$, and denote by $\langle [T] \rangle$ the
subgroup in $\Hom_{d-1}^{\mathscr{S}}(B_0;G)$ generated by $[T]$, then 
we use $\cc(B_0,G,T)$ to denote the collection $\cc(B_0,G,\Psi_{d-1}^{\mathscr{S}}
(\langle [T] \rangle))$.
\begin{theorem} \label{thm:scequ}
	Let $B_0\subseteq \mathbb{R}^{n}$ be a compact Lipschitz neighborhood retract 
	with $\HM^d(B_0)=0$. Let $G$ be a discrete normed abelian group. 
	Then for any $\mathscr{S}\in \{\LC,\NC,\RC,\FC,\NC^c,\RC^c,\FC^c\}$ and
	$T\in\mathscr{S}_{d-1}(B_0;G)$ with $\partial T=0$,	we have that 
	\begin{equation}\label{eq:scequ}
		\inf\{\size(S): \partial S=T,S\in\mathscr{S}_{d}(\mathbb{R}^{n};G)\}
		=\inf\{\HM^{d}(E): E\in \cc(B_0,G,T)\}.
	\end{equation}
\end{theorem}

We denote by $\mathscr{D}^d(\mathbb{R}^n)$ the topological vector space of
smooth differential forms of degree $d$ with compact support in
$\mathbb{R}^{n}$. The elements of dual $\mathscr{D}_d(\mathbb{R}^n)$ are
called $d$-dimensional currents, see \cite[4.1.7]{Federer:1969}. The mass of a
current $T\in \mathscr{D}_d(\mathbb{R}^n)$ is defined by 
\[
	\mass(T)=\sup\{\langle T,\omega \rangle: \omega\in
	\mathscr{D}^d(\mathbb{R}^n), \|\omega\|_{\infty}\leq 1\}.
\]
A $d$-dimensional locally rectifiable current is a current $T\in \mathscr{D}_{d}
(\mathbb{R}^n)$ which can be written as $T=(\HM^d\mr M)
\wedge \eta$, i.e.  
\[
	\langle T,\omega \rangle=\int_{x\in M}\langle \eta(x),\omega \rangle
	d\HM^d(x),\ \omega\in \mathscr{D}^d(\mathbb{R}^n),
\]
where $M$ is a $d$-rectifiable set, $\eta:\mathbb{R}^n\to
\wedge_d\mathbb{R}^n$ is locally summable  such that for $\HM^d\mr M$-a.e. 
$\eta(x)$ is a simple $d$-vector associated with the approximate tangent space 
of $M$ at $x$ and $|\eta(x)|$ is an integer. A rectifiable current is a
locally rectifiable current with compact support. A $d$-dimensional locally
integral current is a $d$-dimensional locally rectifiable current $T$ such that
$\partial T$ is a $(d-1)$-dimensional locally rectifiable current, an integral
current is a locally integral current with compact support, see \cite[4.1.28]
{Federer:1969}. We denote by $\RC_d^{loc}(\mathbb{R}^n)$ and $\IC_d^{loc}
(\mathbb{R}^n)$ the collections of all $d$-dimensional locally rectifiable and
locally integral currents respectively, and by $\RC_d(\mathbb{R}^n)$ and $\IC_d
(\mathbb{R}^n)$ the collections of all $d$-dimensional rectifiable and
integral currents respectively. For any $T\in \mathscr{R}_d^{loc}
(\mathbb{R}^n)$, we define the Hausdorff size of $T$ by 
\[
	\size(T)=\inf\{\HM^d(E): E \text{ are Borel sets such that }T\mr E =T\}.
\]
Since the group $\RC_{d}^{loc}(\mathbb{R}^n)$ of locally rectifiable currents 
and the group $\RC_{d}(\mathbb{R}^n;\mathbb{Z})$ of rectifiable flat 
$\mathbb{Z}$-chains are isometrically isomorphic, see Lemma \ref{le:isoR}, as
a consequence of Theorem \ref{thm:scequ}, we have the following corollary.
\begin{corollary}\label{co:scequ}
	Let $B_0\subseteq \mathbb{R}^{n}$ be a compact Lipschitz neighborhood
	retract with $\HM^d(B_0)=0$.
	Then for any $\mathscr{S}\in\{\IC,\RC,\IC^{loc},\RC^{loc}\}$,
	$\sigma\in \mathscr{S}_{d-1}(\mathbb{R}^n)$ with $\partial \sigma=0$ and 
	$\spt \sigma \subseteq B_0$, we have that
	\begin{equation}\label{eq:scequ00}
		\inf\{\size(T): T\in \mathscr{S}_d(\mathbb{R}^n),\ \partial T=\sigma\}=
		\inf\{\HM^d(E):E\in \cc(B_0,\mathbb{Z},\sigma)\}.
	\end{equation}
\end{corollary}
For any $\mathscr{S}\in \{\IC,\RC,\FC,\mathbf{N}\}$ and sets $B\subseteq X
\subseteq \mathbb{R}^n$, we put 
\begin{equation}\label{eq:hg}
	\begin{gathered}
		\mathcal{Z}_{d-1}^{\mathscr{S}}(X,B)=\{T\in \mathscr{S}_{d-1}(\mathbb{R}^n):\spt T\subseteq X,
		\spt \partial T\subseteq B \text{ or }d=1\},\\
		\mathcal{B}_{d-1}^{\mathscr{S}}(X,B)=\{T+\partial S: T\in \mathscr{S}_{d-1}(\mathbb{R}^n),
		\spt T\subseteq B, S\in \mathscr{S}_{d}(\mathbb{R}^n), \spt S\subseteq X \},
	\end{gathered}
\end{equation}
and define the homology group
\[
	\Hom_{d-1}^{\mathscr{S}}(X,B)=	\mathcal{Z}_{d-1}^{\mathscr{S}}(X,B)/
	\mathcal{B}_{d-1}^{\mathscr{S}}(X,B).
\]
By Theorem 5.11 in \cite{FF:1960} and Proposition 3.7 in \cite{Pauw:2007}, we
see that $\Hom_{\ast}^{\IC}$, $\Hom_{\ast}^{\FC}$ and $\Hom_{\ast}^{\mathbf{N}}$ 
satisfy the seven axioms of Eilenberg and Steenrod \cite{ES:1945,ES:1952}. 
Thus we get the following existence of size minimal integral cycles from the
regularity of minimal sets.
\begin{proposition} \label{prop:com}
	Let $M\subseteq \mathbb{R}^3$ be a 2-dimensional compact submanifold of class
	$C^{1,\alpha}$ without boundary, $0<\alpha\leq 1$. Then for any 
	$\sigma\in \Hom_1^{\IC}(M)$, there exists $[T_0]\in \Hom_2^{\IC}(\mathbb{R}^3,M)$ 
	such that $[\partial T_0]= \sigma$ and 
	\[
		\size(T_0)=\inf\{\size(T):[T]\in \Hom_2^{\IC}(\mathbb{R}^3,M),
		[\partial T]=\sigma\}.
	\]
\end{proposition}
It is still valid if we replace the integral currents homology
$\Hom^{\IC}$ with $\Hom^{\mathbf{N}}$ or $\Hom^{\RC}$ in the above proposition. 

Let's quickly brief readers on the strategy of the paper. Section 2 is devoted
to giving some results that a $d$-set can be deformed into a polyhedral
network and the increment of Hausdorff measure can be  arbitrarily small, see
Theorem \ref{thm:polyapp}. That allows us to deformed a rectifiable flat chain
into polyhedral chain such that the  increment of the size can be
arbitrarily small, see Theorem \ref{thm:pca}. Following from this result, we
will get that the infimum value for size minimizing problem for flat chains is
equal to the infimum value for the corresponding  Plateau's problem with
\v{C}ech homology conditions. In Section 4, we will see that currents and flat
$\mathbb{Z}$ chains are the same. 
%In Section 5, we will see there are natural isomorphism between certain
% homology groups on compact Lipschitz neighborhood retract. 
Theorem \ref{thm:hequ} is following from the deformation theorem which is
developed in Section 2, and the proof can be found in Section 6.

\section{Approximation by polyhedrons}
For any $x\in \mathbb{R}^n$, $r>0$, we denote by $\oball(x,r)$ and
$\cball(x,r)$ the open and the closed balls centered at $x$ with radius $r$
respectively, and define the mapping $\scale{x,r}:\mathbb{R}^n\to
\mathbb{R}^n$ by $\scale{x,r}(y)=x+r\cdot (y-x)$. For any subsets $A,B$ in
$\mathbb{R}^n$, we denote $A+B=\{x+y:x\in A,y\in B\}$, thus
$A+\oball(0,\varepsilon)$ is the $\varepsilon$ neighborhood of $A$ in
$\mathbb{R}^{n}$. We denote by $\omega_n$ and $\sigma_{n-1}$ the
$n$-dimensional Hausdorff measure of the unit ball and the $(n-1)$-dimensional
Hausdorff measure of the unit sphere in $\mathbb{R}^{n}$ respectively. Let
$H\subseteq \mathbb{R}^n$ be a linear subspace. If $d\leq \dim H \leq n$, then
$H$ is linearly isometric to $\mathbb{R}^{\dim H}$, we denote by
$\grass{H}{d}$ and $\gm{H}{d}$ the corresponding Grassmannian  manifold and
the invariant probability measure respectively.

Let $P$ be any convex polyhedron in $\mathbb{R}^n$. Suppose $H$ is the
smallest affine space which contains $P$. We denote by $\mathfrak{r}(P)$ the 
infimum of the radius of balls in $H$ that contains $P$, by $r(P)$ the 
supremum of the radius of balls in $H$ which is contained in $P$, and define 
the rotundity of $P$ by $\mathcal{R}(P)=r(P)/\mathfrak{r}(P)$. Let $\mathcal{K}$
be any polyhedral complex in $\mathbb{R}^n$. We denote by $|\mathcal{K}|$ the
realization of $\mathcal{K}$, and put 
\[
	\mathfrak{r}(\mathcal{K})=\max\{\mathfrak{r}(P):P\in \mathcal{K}\}
	\text{ and } \mathcal{R}(\mathcal{K})=\min\{\mathcal{R}(P):P\in
	\mathcal{K}\}.
\]
For any $x\in \intr(P)$, we denote by $\Pi_{P,x}$ the mapping $P\setminus
\{x\}\to \partial P$ defined by 
\[
	\Pi_{P,x}(y)=\{x+t(y-x):t\geq 0\}\cap \partial P.
\]
\begin{lemma}\label{le:projirrp}
	Let $E\subseteq [0,1]\times \mathbb{R}^{n-1}$ be a $\HM^d$-measurable bounded purely
	$d$-unrectifiable set. Then for $\HM^n$-a.e. $x\in (1,\infty)\times
	\mathbb{R}^{n-1}$, $\Psi_x(E)$ is purely $d$-unrectifiable, where the
	mapping $\Psi_x:[0,1]\times\mathbb{R}^{n-1} \to \{0\}\times
	\mathbb{R}^{n-1}$ is defined by $\Psi_x(z)=\{x+\lambda (z-x):
	\lambda\in \mathbb{R}\}\cap \{0\}\times \mathbb{R}^{n-1}$.
\end{lemma}

\begin{proof}
	Take any closed ball $\cball\subseteq (1,\infty)\times \mathbb{R}^{n-1}$,
	and we will show that for $\HM^n$-a.e. $x\in \cball$, $\Psi_x(E)$ is purely
	$d$-unrectifiable. We put $H=\{0\}\times \mathbb{R}^{n-1}$.  For any $d$-plane $V\subseteq H$ through 0, we denote by
	$V_H^{\perp}$ the space of all vectors in $H$ which are perpendicular to $V$.
	For any vector $u\in \mathbb{R}^n\setminus \{0\}$, we denote by $P(u)$ the 
	vector space generated by $u$ and $V_H^{\perp}$, and denote by $W$ the 
	$(d+1)$-dimensional vector space $V+\{te:t\in \mathbb{R}\}$, where $e=(1,0,
	\cdots,0)\in \mathbb{R}^n$. We put 
	\[
		M(\cball)=\int_{x\in \cball} \int_{V\in \grass{H}{d}}
		\HM^d(V_{\natural}\circ\Psi_x(E)) d \gm{H}{d}(V) d \HM^n(x).
	\]
	We assume $\cball=\cball(x_0,R_0)$. Let $\chi$ be the set function given by
	$\chi(\emptyset)=0$ and $\chi(X)=1$ if $X\neq \emptyset$. Then we get that 
	\[
		\begin{aligned}
			M(\cball)&=\int_{V\in \grass{H}{d}}\int_{x\in \cball} 
			\HM^d(V_{\natural}\circ\Psi_x(E)) d \HM^n(x) d \gm{H}{d}(V) \\
			&=\int_{V\in \grass{H}{d}}\int_{x\in \cball} 
			\int_{v\in V} \chi\big(\Psi_x(E)\cap (v+ V_H^{\perp})\big) d \HM^d(v) d
			\HM^n(x) d \gm{H}{d}(V) \\
			&=\int_{V\in \grass{H}{d}} 
			\int_{v\in V}\int_{x\in \cball} \chi\big(\Psi_x(E)\cap (v+ V_H^{\perp})\big)d
			\HM^n(x)  d \HM^d(v) d \gm{H}{d}(V). 
		\end{aligned}
	\]
	For any $v\in V$, let $f_v:\cball\to W\cap \partial \cball(0,1)$ be the	
	mapping given by	$f_v(x)=W_{\natural}(x-v)/|W_{\natural}(x-v)|$.
	We see that for any $u\in W\cap \partial \cball(0,1)$, $f_v^{-1}(u)= \cball
	\cap (v+P(u))$. If $E\cap
	(v+P(u))=\emptyset$, then for any $x\in f_v^{-1}(u)$, $\Psi_x(E)\cap
	(v+V_H^{\perp})=\emptyset$, thus
	\[
		\int_{y\in f_v^{-1}(u)}\chi\big(\Psi_y(E)\cap (v+ V_H^{\perp})\big) d
		\HM^{n-d}(y)\leq \chi\big(E\cap (v+P(u))\big)\cdot \omega_{n-d} R_0^{n-d}.
	\]
	Since $E\subseteq [0,1]\times \mathbb{R}^{n-1}$ is bounded and
	$\cball(x_0,R_0)\subseteq (1,+\infty)\times \mathbb{R}^{n-1}$, there is a
	ball $\cball(0,R_1)$ such that 
	\[
		\bigcup_{x\in \cball} \Psi_x(E)\subseteq H\cap \cball(0,R_1).
	\]
	Since $\ap J_d f_v(x) \geq 1/|x-v|\geq 1/(|x_0-v|+R_0)\geq 1/C_1$ for any $x\in
	\cball$ and $v\in V\cap \cball(0,R_1)$, and by Theorem 3.2.22 in
	\cite{Federer:1969}, we see that 
	\[
		\begin{aligned}
			&\int_{x\in \cball} \chi\big(\Psi_x(E)\cap (v+ V_H^{\perp})\big) \ap J_d
			f_v(x)d \HM^n(x)\\
			&=\int_{W\cap \partial \cball(0,1)}\int_{y\in f_v^{-1}(u)}
			\chi\big(\Psi_y(E)\cap (v+ V_H^{\perp})\big)d \HM^{n-d}(y) d \HM^d(u)\\
			&\leq \omega_{n-d}R_0^{n-d} \int_{u\in W\cap \partial \cball(0,1)}
			\chi\big(E\cap (v+P(u))\big) d \HM^d(u),
		\end{aligned}
	\]
	thus
	\[
		\begin{aligned}
			\int_{x\in \cball} \chi\big(\Psi_x(E)\cap (v+ V_H^{\perp})\big)d
			\HM^n(x)&\leq C_2 \int_{u\in W\cap \partial \cball(0,1)}
			\chi\big(E\cap (v+P(u))\big) d \HM^d(u),
		\end{aligned}
	\]
	where $C_2=\omega_{n-d}R_0^{n-d} C_1$.
	Hence 
	\[
		\begin{aligned}
			M(\cball)&\leq C_2\int_{V\in \grass{H}{d}}\int_{v\in V}
			\int_{u\in W\cap \partial \cball(0,1)}
			\chi\big(E\cap (v+P(u))\big) d \HM^d(u)d \HM^d(v) d \gm{H}{d}(V)\\
			&= C_2 \int_{V\in \grass{H}{d}}\int_{u\in W\cap \partial
			\cball(0,1)}\HM^d\big(P(u)_{\natural}^{\perp}(E)\big)/|V_{\natural}(u)|
			d \HM^d(u)d \gm{H}{d}(V).
		\end{aligned}
	\]

	Since $E$ is purely unrectifiable, we see that $\HM^d(T_{\natural}(E))=0$
	for $\gm{n}{d}$-a.e. $T\in \grass{n}{d}$, thus
	$\HM^d(T_{\natural}(E))=0$ for $\gm{W}{d}$-a.e. $T\in \grass{W}{d}$. We see
	that $P(u)^{\perp}=\{z\in W:z\perp u\}$. Thus 
	\[
		\int_{u\in W\cap \partial
		\cball(0,1)}\HM^d\big(P(u)_{\natural}^{\perp}(E)\big)
		d \HM^d(u)=\sigma_d\int_{T\in \grass{W}{d}}
		\HM^d(T_{\natural}(E))d \gm{W}{d}(T)=0.
	\]
	Hence
	\[
		\int_{V\in \grass{H}{d}}\int_{u\in W\cap \partial
		\cball(0,1)}\HM^d\big(P(u)_{\natural}^{\perp}(E)\big)
		d \HM^d(u)d \gm{H}{d}(V)=0.
	\]
	 Therefore $M(\cball)=0$, and $\Psi_x(E)$ is purely unrectifiable for
	$\HM^d$-a.e. $x\in \cball$.
\end{proof}
\begin{lemma}\label{le:projirr}
	Suppose $1\leq d< k\leq n$.
	Let $\Delta \subseteq \mathbb{R}^k$ be a closed convex polyhedron of dimension
	$k$. If $E\subseteq \Delta$ is  purely $d$-unrectifiable, $\HM^d$-measurable, 
	and satisfying that $\HM^k(\overline{E})=0$, then for $\HM^k$-a.e. $x\in \intr(\Delta)$, 
	$\Pi_{\Delta,x}(E)$ is purely $d$-unrectifiable.
\end{lemma}
\begin{proof}
	Let $A$ be the smallest affine space which contains $\Delta$. We put
	$O=\intr(\Delta)\setminus \overline{E}$. For any $x\in O$, there is a ball
	$\oball(x,r)$ such that $A\cap \oball(x,r)\subseteq O$. Since $O$ is open in
	$A$, there is positive number $\delta>0$ such that
	$\scale{x,\delta}(\Delta)\subseteq O$. We assume $A\cap
	\oball(x,r_1)\subseteq \Delta$, $r_1>0$, then $A\cap \oball(x,\delta
	r_1)\subseteq \scale{x,\delta}(\Delta)$. By Lemma \ref{le:projirrp}, for 
	$\HM^k$-a.e. $x\in \oball(x,\delta r_1)$, $\Pi_{\Delta,x}(E)$ is purely
	$d$-unrectifiable. Therefore, $\Pi_{\Delta,x}(E)$ is purely
	$d$-unrectifiable for $\HM^k$-a.e. $x\in O$. Since $\HM^k(\overline{E})=0$,
	we get that $\Pi_{\Delta,x}(E)$ is purely $d$-unrectifiable for $\HM^k$-a.e.
	$x\in \intr(\Delta)$.
\end{proof}
\begin{lemma}\label{le:projrec}
	Suppose $1\leq d< k\leq n$. There exists a constant $C=C(k,d)>0$ such that 
	for any closed convex polyhedron $\Delta$ of dimension $k$, $\HM^d$-measurable
	set $E\subseteq \Delta$, and $0<\beta<1$, we can find a ball
	$\cball_{\Delta}=\cball(x_0,r_0)$  and a set $Y_{\Delta}\subseteq 
	\cball(x_0,r_0)$ satisfying that $r_0\geq \mathfrak{r}(\Delta)\mathcal{R}
	(\Delta)/4$, $\cball(x_0,2r_0)\subseteq \Delta$, $\HM^k(Y_{\Delta})\geq
	(1-\beta) \omega_k r_0^k$, and for any $x\in Y_{\Delta}$
	\begin{equation}\label{eq:projrec1}
		\HM^d(\Pi_{\Delta,x}(E))\leq C \beta^{-1}\cdot 
		\mathcal{R}(\Delta)^{-2d}\HM^d(E).
	\end{equation}
\end{lemma}
\begin{proof}
	Without loss of generality, we assume that $k=n$, $\diam(\Delta)\leq 1$ and
	$0<\HM^d(E)<+\infty$. 
	Take	$\cball(x_1,2r_0)\subseteq \Delta\subseteq \cball(x_2,R_0)$ such
	that $2r_0/R_0\geq \mathcal{R}(\Delta)/2$. For any $x\in \cball(x_1,r_0)$, 
	we define the mapping $p_x:\Delta\setminus \{x\}\to \mathbb{R}^n\cap 
	\partial \cball(0,1)$ by 
	\[
		p_x(z)=\frac{z-x}{|z-x|}.
	\]
	Then $p_x\vert_{\partial \Delta}: \partial \Delta \to \partial
	\cball(0,1)$ is biLipschitz, and $\Pi_{\Delta,x}=(p_x\vert_{\partial
	\Delta})^{-1}\circ p_x$. Since $\Lip(p_x\vert_{\partial \Delta})^{-1}\leq
	8R_0^2/r_0$ and $\|D p_x(z)\|=1/|z-x|$, we get that
	\[
		\HM^d(p_x(E))\leq \int_{z\in E}\|D p_x(z)\|^d d \HM^d(z)\leq \int_{z\in E}\frac{1}{|z-x|^d} d \HM^d(z),
	\]
	\[
		\HM^d(\Pi_{\Delta,x}(E))\leq \left(\Lip(p_x\vert_{\partial \Delta})^{-1}\right)^d
		\HM^d(p_x(E))\leq (8R_0^2/r_0)^d\int_{z\in E}\frac{1}{|z-x|^d} d
		\HM^d(z),
	\]
	and 
	\[
		\int_{x\in \cball(x_0,r_0)}\HM^d(\Pi_{\Delta,x}(E)) d \HM^k(x)\leq
		(8R_0^2/r_0)^d\int_{z\in E}\int_{x\in \cball(x_0,r_0)}\frac{1}{|z-x|^d}
		d \HM^k(x) d \HM^d(z).
	\]
	Since there exist a constant $C_1>0$ which only depends on $k$ and $d$ such
	that
	\[
		\int_{x\in \cball(x_0,r_0)}\frac{1}{|z-x|^d} d \HM^k(x)\leq C_1
		r_0^{k-d},
	\]
	we get that 
	\[
		\int_{x\in \cball(x_0,r_0)}\HM^d(\Pi_{\Delta,x}(E)) d \HM^k(x)\leq
		C_1(8R_0^2/r_0)^dr_0^{k-d}\HM^d(E)\leq 2^{7d}C_1
		\mathcal{R}(\Delta)^{-2d}\HM^d(E) r_0^{k}.
	\]
	Appling the Chebyshev's inequality, we get that 
	\[
		\HM^k(\{x\in \cball(x_0,r_0):\HM^d(\Pi_{\Delta,x}(E))\leq C
		\beta^{-1}\cdot \mathcal{R}(\Delta)^{-2d}\HM^d(E)\})\geq
		(1-\beta)\omega_k r_0^k,
	\]
	where $C=2^{7d}\omega_k^{-1} C_1$.
\end{proof}
\begin{lemma}\label{le:intden}
	Let $F:\mathbb{R}^n\times\grass{n}{d}\to \mathbb{R}_{+}$ be an integrand of
	class Baire 1, $\lambda:2^{\mathbb{R}^n}\to \mathbb{R}$ be any positive set
	function, and $X\subseteq \mathbb{R}^n$ be any $\HM^d$-measurable
	$d$-rectifiable set with $\HM^d(X)<\infty$ and $\Phi_{F,\lambda}(X)<\infty$.
	Then for $\HM^d$-a.e. $x\in X$,
	\begin{equation}\label{eq:intden}
		\lim_{r\to 0+}\frac{\Phi_{F,\lambda}(X\cap \cball(x,r))}{\omega_d r^d\cdot
		F(x,\Tan(X,x))}=1.
	\end{equation}
\end{lemma}
\begin{proof}
	Since $X$ is $d$-rectifiable and $\HM^d(X)<\infty$, we get that the function
	$X\to X\times \grass{n}{d}$ given by $x\mapsto (x,\Tan(X,x))$ is
	$\HM^d$-measurable. Since $F$ is of class Baire 1, we get that the
	function $X\to \mathbb{R}$ given by $x\mapsto F(x,\Tan(X,x))$ is
	$\HM^d$-measurable. By Corollary 2.14 in \cite{Mattila:1995}, we get that
	\eqref{eq:intden} holds for $\HM^d$-a.e. $x\in X$.
\end{proof}
\begin{lemma} \label{le:mp}
	Let $A\subseteq \oball(x,r)$ be a convex set. Let $P$ be a $d$-plane
	through $x$. Suppose that $0<\varepsilon, \delta<1/2$, $A+\cball(0,\delta
	r)\subseteq \oball(x,r)$ and $A\subseteq P+\cball(0,\varepsilon r)$. Then 
	there is a Lipschitz mapping $\varphi:\mathbb{R}^n\to \mathbb{R}^n$ such that 
	\[
		\varphi\vert_{\mathbb{R}^n\setminus
		\oball(x,r)}=\id_{\mathbb{R}^n\setminus \oball(x,r)},\
		\varphi\vert_{A}=P_{\natural}\vert_A, \varphi(\oball(x,r))\subseteq
		\oball(x,r)\text{ and }\Lip(\varphi)\leq 3+\varepsilon/\delta.
	\]
\end{lemma}
\begin{proof}
	Let $\eta:\mathbb{R}\to \mathbb{R}$ be the function defined by 
	\[
		\eta(t)=\begin{cases}
			1,&t\leq 0,\\
			1-t/\delta,& 0<t\leq \delta,\\
			0,& t> \delta.
		\end{cases}
	\]
	Then $\eta$ is Lipschitz with $\Lip(\eta)=1/\delta$.
	Let $\varphi:\mathbb{R}^n\to \mathbb{R}^n$ be the mapping defined by 
	\[
		\varphi(z)=z-\eta(\dist(z,A)/r)P_{\natural}^{\perp}(z),\ \forall z\in
		\mathbb{R}^n.
	\]
	For any $z_1,z_2\in \mathbb{R}^n$, if $\min\{\dist(z_1,A),\dist(z_2,A)\}\geq
	\delta r$, then $\varphi(z_1)=z_1$ and $\varphi(z_2)=z_2$, thus
	$|\varphi(z_1)-\varphi(z_2)|=|z_1-z_2|$.
	If $\min\{\dist(z_1,A),\dist(z_2,A)\}< \delta r$, we assume that
	$\dist(z_2,A)< \delta r$, then $|P_{\natural}^{\perp}(z_2)|\leq
	(\varepsilon+\delta)r$ and
	\[
		\begin{aligned}
			\varphi(z_1)-\varphi(z_2)&=(z_1-z_2)-
			\eta(\dist(z_1,A)/r)P_{\natural}^{\perp}(z_1-z_2)\\
			&\quad+
			P_{\natural}^{\perp}(z_2)(\eta(\dist(z_1,A)/r)-\eta(\dist(z_2,A)/r)),
		\end{aligned}
	\]
	thus 
	\[
		\begin{aligned}
			|\varphi(z_1)-\varphi(z_2)|&\leq 2|z_1-z_2|+|
			P_{\natural}^{\perp}(z_2)| \Lip(\eta) \cdot \frac{1}{r}
			|\dist(z_1,A)-\dist(z_2,A)|\\
			&\leq (3+\varepsilon/\delta)|z_1-z_2|.
		\end{aligned}
	\]
	We get that $\varphi$ is Lipschitz and $\Lip(\varphi)\leq 3+\varepsilon/\delta$. 
\end{proof}
%The following theorem essentially adopt a theorem of V. Feuvrier, see 
%Th\'eor\`em 4.3.17 in \cite{Feuvrier:2008} or Theorem 3 in
%\cite{Feuvrier:2009}, however the original version was proved only for Hausdorff
%measure, and here we generalize the statement from Hausdorff measure to
%certain integrands. 
\begin{theorem}
	\label{thm:polyapp}
	Let $F_{\ell}:\mathbb{R}^n\times \grass{n}{d}\to \mathbb{R}_+$, $\ell\in
	\mathbb{Z}\cap [1,\ell_0]$, be a family of integrands of class Baire 1 and 
	bounded above. Let $\lambda_{\ell}$, $\ell\in \mathbb{Z}\cap [1,\ell_0]$, be
	positive set functions bounded above.
	Let $U\subseteq \mathbb{R}^n$ be an open set. Then there exist constants 
	$c_1=c_1(n)>0$, and $c_2=c_2(n,d)>0$ such that for any $\varepsilon_1>0$, $\varepsilon_2>0$, compact set
	$K\subseteq U$, and $\HM^d$-measurable set $X$ satisfying that $\HM^d(X\cap
	U)<\infty$ and $\HM^{d+1}(\overline{X}\cap U)=0$, we 
	can find a polyhedral complex $\mathcal{K}$ and a Lipschitz mapping
	$\varphi:\mathbb{R}^n\to \mathbb{R}^n$ satisfying that
	\begin{enumerate}[label*=(\arabic*),ref=\emph{(\arabic*)}]
		\item \label{polyapp1}$\mathcal{R}(\mathcal{K})\geq c_1,
			\mathfrak{r}(\mathcal{K})\leq \varepsilon_1$, $|\mathcal{K}|\supseteq K$;
		\item \label{polyapp2} $\varphi\vert_{\mathbb{R}^n\setminus U}=
			\id_{\mathbb{R}^n\setminus U}$, $\varphi(U)\subseteq U$, 
			$\varphi\vert_{U}$ is homotopic to $\id_U$,
			and $\|\varphi-\id\|_{\infty}\leq \varepsilon_1$;
		\item \label{polyapp3} $\varphi(X)\cap |\mathcal{K}|\subseteq
			|\mathcal{K}_d|$ and $|\mathcal{K}_d|\setminus
			\overline{\varphi(X)}=\emptyset$ for
			some sub-complex $\mathcal{K}_d$ with $\dim(\mathcal{K}_d)\leq d$;
		\item \label{polyapp4} $\Phi_{F_{\ell},\lambda_{\ell}}\left(\varphi(X\cap
			U)\right)\leq \Phi_{F_{\ell},\lambda_{\ell}}(X_{rec})+\varepsilon_2$, 
			where $X_{rec}$ is the $d$-rectifiable part of $X\cap U$;
		\item \label{polyapp5} $\HM^d(\varphi(X\cap P))\leq c_2\HM^d(X\cap P)$ for any
			$P\in \mathcal{K}$. 
	\end{enumerate}
\end{theorem}
\begin{proof}
	Since $F_{\ell}$ and $ \lambda_{\ell}$ are bounded above, we assume that
	$F_{\ell}\leq b <\infty$ and $\lambda_{\ell}\leq b<+\infty$ for any $1\leq 
	\ell \leq \ell_0$. We fix a small number $\varepsilon>0$, and chose it later.
	We take compact set $K_0$ and positive number $0<\delta_0< \min\{
	\varepsilon_1, 1\}$ such that $K\subseteq K_0\subseteq K_0+
	\oball(0,\delta_0)\subseteq U$ such that
	\[
		\HM^d(X\cap U\setminus K_0 )\leq \varepsilon.
	\]
	Let $\{C_i\}_{i\in I}$ be the family of finite many standard dyadic cubes with
	sidelength $2^{-k}\leq (100\sqrt{n})^{-1}\delta_0$ such that 
	\[
		8C_i\cap K_0\neq \emptyset,\ \forall i\in I.
	\]
	We write $X\cap U=X_{rec}\cup X_{irr}$, where $X_{rec}$ is $d$-rectifiable and
	$\HM^d$-measurable, $X_{irr}$ is purely $d$-unrectifiable and $\HM^d$-measurable. 
	By Lemma \ref{le:intden} and the existence of approximate tangent planes almost
	everywhere for rectifiable sets, we get that for $\HM^d$-a.e. $x\in X_{rec}$,
	there exist $d$-plane $P_x$ and $r_{\varepsilon}(x)>0$ such that 
	$\oball(x,2r_{\varepsilon}(x)) \subseteq U$ and for any $0<r<r_{\varepsilon}(x)$,
	\[
		\left|\frac{\Phi_{F_{\ell},\lambda_{\ell}}(X_{rec}\cap \cball(x,r))}{\omega_d r^d
		F_{\ell}(x,P_x)}-1\right|\leq \varepsilon,\ 
		\left|\frac{\HM^d(X_{rec}\cap \cball(x,r))}{\omega_d r^d}-1\right|\leq
		\varepsilon,
	\]
	\[
		\frac{\HM^d(X_{rec}\cap \cball(x,r)\setminus
		\mathcal{C}(x,P_x, r,\varepsilon))}{\omega_d r^d}\leq \varepsilon,\
		\frac{\HM^d(X_{irr}\cap\cball(x,r))}{\omega_d r^d}\leq \varepsilon,
	\]
	where $\mathcal{C}(x,P_x,r, \varepsilon)=\{z\in \cball(x,r):
	\dist(z-x,P_x)\leq \varepsilon |z-x|\}$, we denote by $X_{\varepsilon}$ all of such 
	points $x\in X_{rec}$. Then we see that 
	\[
		\left\{\cball(x,r):x\in X_{\varepsilon},
		0<r<\min\left\{r_{\varepsilon}(x),2^{-k}\right\}\right\} 
	\]
	is a Vitali covering of $X_{\varepsilon}$. By Vitali covering theorem,we can
	find a countable family disjoint balls $\{\cball(x_j,r_j)\}_{j\in J}$ 
	such that 
	\[
		\HM^d\left(X_{\varepsilon}\setminus \bigcup_{j\in J}\cball(x_j,r_j)\right)=0.
	\]
	Thus we can find finite many disjoint balls $\{\cball_j=\cball(x_j,r_j)\}_{1\leq
	j\leq	N}$ such that 
	\[
		\HM^d\left(X_{\varepsilon}\setminus \bigcup_{j=1}^{N}\cball_j\right)\leq
		\varepsilon .
	\]
	We put $P_j=P_{x_j}$, $A_j=\cball_j\cap (P_j+\cball(0,\varepsilon))$,
	$A_j'=\cball(x_j,(1-2 \varepsilon) r_j)\cap (P_j+\cball(0, \varepsilon r_j))$
	and $\oball_j=\oball(x_j,r_j)$. By Lemma \ref{le:mp}, there is a Lipschitz
	mapping	$\psi_j:\mathbb{R}^n\to \mathbb{R}^n$ such that 
	\begin{equation}\label{eq:polyapp105}
		\psi_j\vert_{\mathbb{R}^n\setminus \oball_j}=\id_{\mathbb{R}^n\setminus
		\oball_j}, \psi_j\vert_{A_j'}=(P_j)_{\natural}\vert_{A_j'},
		\psi_j(\oball_j)\subseteq \oball_j \text{ and }\Lip(\psi_j)\leq 4.
	\end{equation}
	We put $\psi=\psi_N\circ\cdots\circ \psi_1$. Since $\cball_j$, $1\leq j\leq N$,
	are disjoint, we have that $\Lip(\psi)\leq 4$,
	$\|\psi-\id\|_{\infty}\leq \max\{r_j:1\leq j\leq N\}\leq 2^{-k}$,
	$\psi(\cball_j)\subseteq \cball_j$ and $\psi(x)=x$ for $x\in \mathbb{R}^n
	\setminus \cup_{1\leq j\leq N}\cball_j$. Thus
	\begin{equation}\label{eq:polyapp120}
		\begin{aligned}
			\HM^d(X_{irr}\sd \psi(X_{irr}))&\leq \left(1+\Lip(\psi)^d\right)
			\sum_{j=1}^{N} \HM^d(X_{irr}\cap \cball_j)\leq (1+4^d)
			\varepsilon\sum_{j=1}^{N}\omega_dr_j^{d}\\
			&\leq (1+4^d)\frac{\varepsilon}{1-\varepsilon}
			\sum_{j=1}^{N}\HM^d(X_{rec}\cap \cball_j)\leq 2(1+4^d) 
			\HM^d(X_{rec})\varepsilon.
		\end{aligned}
	\end{equation}
	Since $\psi\vert_{A_j'}=\psi_j\vert_{A_j'}=(P_j)_{\natural}\vert_{A_j'}$, we
	get that
	\begin{equation}\label{eq:polyapp122}
		\HM^d(X_{rec}\cap \cball_j\setminus A_j')\leq (1+\varepsilon) \omega_d
		r_j^{d} -(1-\varepsilon)\omega_d ((1-2 \varepsilon)r_j)^d + \varepsilon
		\omega_d r_j^d\leq 2(d+1) \varepsilon \omega_d r_j^d
	\end{equation}
	and 
	\[
		\begin{aligned}
			\Phi_{F_{\ell},\lambda_{\ell}}(\psi(X_{rec}\cap \cball_j))&\leq
			\Phi_{F_{\ell},\lambda_{\ell}}(\psi(X_{rec}\cap A_{j'}))+
			\Lip(\psi)^db\HM^d(X_{rec}\cap \cball_j\setminus A_j')\\
			&\leq \omega_d(1-2 \varepsilon)^d r_j^d F(x_j,P_j)+2\cdot 4^d (d+1)b
			\varepsilon \omega_d r_j^d \\
			&\leq \Phi_{F_{\ell},\lambda_{\ell}}(X_{rec}\cap \cball_j)+ 2\cdot 4^d (d+1)
			b\varepsilon \omega_d r_j^d,
		\end{aligned}
	\]
	by setting $B=\cup_{1\leq j\leq N}\cball_j$, $\eta_1=2\cdot 4^d
	(d+1)b$, we have that $\psi(B)\subseteq B$, $\psi(x)=x$ for $x\in
	\mathbb{R}^n\setminus B$, and
	\begin{equation}\label{eq:polyapp130}
		\Phi_{F_{\ell},\lambda_{\ell}}(\psi(X_{rec}\cap B))\leq
		\Phi_{F_{\ell},\lambda_{\ell}}(X_{rec}\cap B) +\eta_1 \varepsilon
		\sum_{j=1}^N \omega_d r_j^{d}\leq \Phi_{F_{\ell},\lambda_{\ell}}(X_{rec})
		+2\eta_1 \HM^d(X_{rec})\varepsilon.
	\end{equation}

	For any $1\leq j\leq N$, let $\{S_{j,i}\}_{i\in I_j}$ be a family of finite 
	many dyadic cubes which are parallel to $P_j$ and of sidelength $2^{-k_j}\leq
	(100\sqrt{n})^{-1}\varepsilon r_j$ and have nonempty intersection with
	$A_j'$. We put $\mathcal{O}=\cup_{i\in I}C_i$ and $k_0=\max\{k_j:1\leq j\leq N\}$. 
	Let $\{S_{0,i}\}_{i\in I_0}$ be the collection of standard dyadic cubes $C$
	of sidelength $2^{-k_0}$ such that the $k$-th ancestor $\widehat{C}$ of
	$C$ is contained in $\{C_i\}_{i\in I}$ and $\widehat{C}\cap \cball_j=
	\emptyset$ for any $1\leq j\leq N$. Let $\mathcal{K}_j$, 
	$0\leq j\leq N$ be the polyhedral complex associate with 
	$\{S_{j,i}\}_{i\in I_j}$. Applying Th\'eor\`eme 1 in \cite{Feuvrier:2012}, 
	there is a polyhedral complex $\mathcal{K}'$ and a
	constant $c_1=c_1(n)>0$ such that $\mathcal{K}_j$, $1\leq j\leq N$, are 
	subcomplexes of $\mathcal{K}'$, $|\mathcal{K}'|=\mathcal{O}$,
	$\mathcal{R}(\mathcal{K}')\geq c_1$, and $\mathfrak{r}(\mathcal{K}')\leq
	\sqrt{n}\max\{2^{-k_j}:0\leq j\leq N\}$.
	For any $\Delta\in \mathcal{K}'$, we denote $\mathring{\Delta}=\Delta\setminus
	\cup\{\beta\in \mathcal{K}',\dim(\beta)<\dim(\Delta)\}$. We denote by
	$\mathcal{K}^{(\ell)}$ the collection of all simplexes $\Delta\in
	\mathcal{K}'$ such that $\dim(\Delta)=\ell$ and $\intr(\Delta)\cap \partial
	\mathcal{O}=\emptyset$. We set
	%\begin{equation}\label{eq:polyapp150}
	%\end{equation}
	$\mathcal{K}=\{\Delta\in \mathcal{K}':
	\Delta\cap \partial \mathcal{O}=\emptyset\}$. Then \ref{polyapp1} holds for
	complex $\mathcal{K}$.

	We put $E=\psi(X)$. Then $E$ is also
	$\HM^d$-measurable. Since $\psi$ is Lipschitz and $\{x\in
	\mathbb{R}^n:\psi(x)\neq x\}$ is bounded, we get that
	$\psi(\overline{X})=\overline{\psi(X)}$ and $\HM^{d+1}(\overline{E}\cap
	U)=0$. We write $E\cap U=E_{rec}\cup E_{irr}$, where
	$E_{irr}\subseteq \psi(X_{irr})$ is purely $d$-unrectifiable and
	$\HM^d$-measurable, $E_{rec}$ is $d$-rectifiable and $\HM^d$-measurable.
	Then 
	\[
		\HM^d(E_{irr}\cap \mathcal{O})\leq \HM^d(X_{irr}\cap \mathcal{O})+\eta_1\HM^d(X_{irr}) \varepsilon,
	\] 
	and 
	\[
		\Phi_{F_{\ell},\lambda_{\ell}}(E_{rec}\cap \mathcal{O})\leq
		\Phi_{F_{\ell},\lambda_{\ell}}(X_{rec}\cap \mathcal{O})+2\eta_1\HM^d(X_{rec}) \varepsilon.
	\]

	Applying Lemma \ref{le:projrec} with $\beta=3/4$, there is a constant
	$\alpha_1=\alpha_1(n,d)>0$ such that for any $\Delta\in \mathcal{K}^{(n)}$, 
	$x\in Y_{\Delta}\setminus \overline{E}$ and any $\HM^d$-measuable set
	$E'\subseteq E$, 
	\begin{equation}\label{eq:polyapp200}
		\HM^d(\Pi_{\Delta,x}(E'\cap \Delta))\leq
		\alpha_1c_1^{-2d}\HM^d(E'\cap \Delta).
	\end{equation}
	By Lemma \ref{le:projirr}, we get that $\Pi_{\Delta,x}(E_{irr}\cap \Delta)$ is
	purely $d$-unrectifiable for $\HM^n$-a.e. $x\in Y_{\Delta}$. So we pick one
	point $x_{\Delta}\in Y_{\Delta}\setminus \overline{E}$ such that  
	$\Pi_{\Delta,x_{\Delta}}(E_{irr} \cap \Delta)$ is purely $d$-unrectifiable 
	and \eqref{eq:polyapp200} holds.

	Since $\mathring{\Delta}\setminus \overline{E}$ is open, we can find small 
	open ball $\oball_{\Delta}=\oball(x_{\Delta}, r_{x_{\Delta}})$ such that
	$\oball_{\Delta}\cap \Delta\subseteq \mathring{\Delta}\setminus \overline{E}$.
	By Lipschitz extension theorem, we can find a Lipschitz mapping
	$\varphi_1:\mathbb{R}^n\to \mathbb{R}^n$ such that 
	\[
		\varphi_1(z)=\begin{cases}
			z,& z\in \mathbb{R}^n\setminus \mathcal{O},\\
			\Pi_{\Delta,x_{\Delta}}(z), &z\in \Delta\setminus \oball_{\Delta},\
			\Delta\in \mathcal{K}^{(n)}.
		\end{cases}
	\]
	Then $\varphi_1(E_{irr})$ is  purely $d$-unrectifiable. Since
	$\Pi_{\Delta,x_{\Delta}}=\id$ on $\Delta\setminus \mathring{\Delta}$, from
	\eqref{eq:polyapp122} and \eqref{eq:polyapp200}, we get that
	\[
		\Phi_{F_{\ell},\lambda_{\ell}}(\varphi_1(E_{rec}\cap \mathcal{O}))\leq
		\Phi_{F_{\ell},\lambda_{\ell}}(E_{rec}\cap \mathcal{O})+ \eta_1 \alpha_1 
		c_1^{-2d}\HM^d(E_{rec})
		\varepsilon.
	\]

	Similar to the argument above, there is a constant $\alpha_2=\alpha_2(n,d)$
	such that, for any $\Delta\in \mathcal{K}^{(n-1)}$, we can find
	$\oball_{\Delta}=\oball(x_{\Delta},r_{x_{\Delta}})$ such that $\Delta\cap
	\oball_{\Delta}\subseteq \mathring{\Delta}\setminus \overline{E}$,
	$\Pi_{\Delta,x_{\Delta}}(\psi(E_{irr}))$ is purely $d$-unrectifiable and 
	\[
		\HM^d(\Pi_{\Delta,x_{\Delta}}(\psi(E')\cap \Delta))\leq
		\alpha_2c_1^{-2d}\HM^d(\psi(E')\cap \Delta) 
	\]
	for any $\HM^d$-measurable set $E'\subseteq E$. There is also a Lipschitz 
	mapping $\varphi_2:\mathbb{R}^n\to \mathbb{R}^n$ such that 
	\[
		\varphi_2(z)=\begin{cases}
			z,& z\in \mathbb{R}^n\setminus \mathcal{O},\\
			\Pi_{\Delta,x_{\Delta}}(z), &z\in \Delta\setminus \oball_{\Delta},\
			\Delta\in \mathcal{K}^{(n-1)}.
		\end{cases}
	\]
	By a similar procedure as above, we can find constants
	$\alpha_3,\cdots,\alpha_{n-d}$ and Lipschitz mappings $\varphi_3,\cdots,
	\varphi_{n-d}$. Finally, by setting $g=\varphi_{n-d}\circ\cdots\circ
	\varphi_1$, then the following hold:
	\begin{itemize}
		\item $g(E\cap \Delta)\subseteq \cup \mathcal{K}^{(d)}$ for any 
			$\Delta\in \mathcal{S}$,
		\item $\HM^d(g(E'\cap \Delta))\leq \alpha\HM^d(E'\cap
			\Delta)$ for any $E'\subseteq E$ and $\Delta\in \mathcal{S}$,
		\item $g(E_{irr}\cap \Delta)$ is 
			purely $d$-unrectifiable for any $\Delta\in
			\mathcal{K}^{(n-d)}\cup \cdots\cup \mathcal{K}^{(n)}$,
		\item $\Phi_{F_{\ell},\lambda_{\ell}}(g(E_{rec}))\leq
			\Phi_{F_{\ell},\lambda_{\ell}}(E_{rec})+\eta_2 \HM^d(E_{rec})
			\varepsilon$,
	\end{itemize}
	where $\mathcal{S}=\{\Delta\in \mathcal{K}:\dim(\Delta)\geq d+1\}$, $\alpha=\alpha_1\cdots \alpha_{n-d}
	c_1^{-2d(n-d)}$, and $\eta_2=\eta_1 \alpha$.

	For any $\Delta\in \mathcal{K}^{(d)}$, if $\mathring{\Delta}\cap
	g(E)\neq \emptyset$ and $\mathring{\Delta}\setminus \overline{g(E)}
	\neq \emptyset$, then we pick a point $x_{\Delta}\in\mathring{\Delta}
	\setminus\overline{ g(E)}$ and a ball
	$\oball_{\Delta}=\oball(x_{\Delta},r_{x_{\Delta}})$ such that
	$\oball_{\Delta}\cap g(E)=\emptyset$,
	and denote by $\mathcal{S}^{(d)}$ the colloection of all such facest $\Delta\in
	\mathcal{K}^{(d)}$. We let $\varphi_{n-d+1}$ be a Lipschitz 
	mapping such that 
	\[
		\varphi_{n-d+1}(z)=\begin{cases}
			z,&z\in \mathbb{R}^n\setminus \mathcal{O},\\
			\Pi_{\Delta,x_{\Delta}}(z),& z\in \Delta\setminus \oball_{\Delta},\
			\Delta\in \mathcal{S}^{(d)},
		\end{cases}
	\]
	and put $\varphi=\varphi_{n-d+1}\circ g\circ\psi$, $c_2=4^d
	\alpha$.

	By the construction of $\varphi_{i}$, $1\leq i\leq n-d+1$, we may assume
	that $\|\varphi_i-\id\|_{\infty}\leq \mathfrak{r}(\mathcal{K}')$ and
	$\|\varphi_{n-d+1}\circ g-\id\|_{\infty}\leq \mathfrak{r}(\mathcal{K}')$,
	we get that $\|\varphi-\id\|_{\infty}\leq \max\{r_j:1\leq j\leq N\}\leq
	\varepsilon_1$, $\varphi(x)=x$ for $x\in \mathbb{R}^n\setminus \mathcal{O}$,
	$\varphi(\mathcal{O})\subseteq \mathcal{O}$. Thus \ref{polyapp2} holds for
	Lipschitz mapping $\varphi$. Since $g(E\cap \Delta)\subseteq \cup
	\mathcal{K}^{(d)}$ for any $\Delta\in \mathcal{S}$, by the construction 
	$\varphi_{n-d+1}$, we get that \ref{polyapp3} holds. We put $D=\cup 
	\mathcal{K}$. Then $K_0\subseteq D\subseteq \mathcal{O}\subseteq
	K_0+\oball(0,\delta_0)$ and $g(D\cap E_{irr})\subseteq g(D\cap E) 
	\subseteq \cup \mathcal{K}^{(d)}$, thus $g(D\cap E_{irr})$ is 
	$d$-rectifiable. But $g(E_{irr}\cap \Delta)$ is purely $d$-unrectifiable
	for any $\Delta\in \mathcal{K}^{(n-d)}\cup \cdots\cup \mathcal{K}^{(n)}$, we
	get that $\HM^d(g(D\cap E_{irr}))=0$, and
	\[
		\HM^d(g(E_{irr}\cap \mathcal{O}))=\HM^d(g(E_{irr}\cap
		\mathcal{O}\setminus D)) \leq \alpha\HM^d(E\cap
		\mathcal{O}\setminus D)\leq \alpha \varepsilon.
	\]
	Since 
	\[
		\begin{aligned}
			\Phi_{F_{\ell},\lambda_{\ell}}(g(E_{rec}\cap \mathcal{O}))&\leq
			\Phi_{F_{\ell},\lambda_{\ell}}(E_{rec}\cap \mathcal{O})+
			\eta_2\HM^d(E_{rec})\varepsilon\\
			&\leq \Phi_{F_{\ell},\lambda_{\ell}}(X_{rec}\cap \mathcal{O})+(\eta_1+
			4^d\eta_2)\HM^d(X_{rec})\varepsilon,
		\end{aligned}
	\]
	we get that 
	\[
		\begin{aligned}
			\Phi_{F_{\ell},\lambda_{\ell}}(\varphi(X\cap U))&\leq
			\Phi_{F_{\ell},\lambda_{\ell}}(g\circ\psi(X\cap U))=
			\Phi_{F_{\ell},\lambda_{\ell}}(g(E\cap U))\\
			&\leq \Phi_{F_{\ell},\lambda_{\ell}}(X\cap U \setminus \mathcal{O})+
			\Phi_{F_{\ell},\lambda_{\ell}}(g(E\cap \mathcal{O}))\\
			&\leq b \varepsilon + \Phi_{F_{\ell},\lambda_{\ell}}(X_{rec}\cap \mathcal{O})+(\eta_1+
			4^d\eta_2)\HM^d(X_{rec})\varepsilon +b \alpha \varepsilon\\
			&\leq \Phi_{F_{\ell},\lambda_{\ell}}(X_{rec})+
			(b+b \alpha+(\eta_1+4^d\eta_2)\HM^d(X\cap U))\varepsilon.
		\end{aligned}
	\]
	If we take $\varepsilon>0$ such that $(b+b
	\alpha+(\eta_1+4^d\eta_2)\HM^d(X\cap U))\varepsilon \leq \varepsilon_2$,
	then we get that 
	\[
		\Phi_{F_{\ell},\lambda_{\ell}}(\varphi(X\cap U))\leq
		\Phi_{F_{\ell},\lambda_{\ell}}(X_{rec})+\varepsilon_2.
	\]
	The inequality \ref{polyapp5} follows from \eqref{eq:polyapp105}
	and $\HM^d(g(E'\cap \Delta))\leq \alpha\HM^d(E'\cap \Delta)$. 

\end{proof}

\begin{remark}
	If $\ell_0=1$, $F_1\equiv 1$ and $\lambda_1\equiv 1$, then
	$\Phi_{1,1}(E)=\HM^d(E)$ for any $d$-dimensional set $E\subseteq
	\mathbb{R}^n$, and the inequality \ref{polyapp4} can be written as 
	\[
		\HM^d(\varphi(X\cap U))\leq \HM^d(X_{rec})+ \varepsilon_2.
	\]
\end{remark}

\section{Flat $G$-chains}
A subset $X\subseteq \mathbb{R}^n$ is called a neighborhood retract if there is 
an open set $U\supseteq X$ and a continuous mapping $\rho:U\to X$ such that
$\rho\vert_X=\id_X$. If $\rho$ is Lipschitz, then we say that $X$ is a Lipschitz
neighborhood retract. If $U=\mathbb{R}^n$ and $\rho$ is Lipschitz, then $X$ is
called a Lipschitz retract.

Let $G$ be a complete abelian group. That is, an abelian group $G$ equipped 
with a norm $|\cdot|:G\to [0,\infty)$ satisfying that 
\begin{itemize}
	\item $|-g|=|g|$ for any $g\in G$;
	\item $|g+h|\leq |g|+|h|$ for any $g,h\in G$;
	\item $|g|\geq 0$ with equality if and only if $g=0$,
\end{itemize}
and which is a complete metric space with respect to the induced metric.

For any $A\in \FC_{m}(\mathbb{R}^n;G)$, there is a sequence of polyhedral
chains $\{P_k\}\subseteq \PC_{m}(\mathbb{R}^n;G)$ such that $\fn(P_k-A)\to 0$. Since
$\fn(\partial P_k-\partial P_{\ell})\leq \fn(P_k-P_{\ell})$, we see that
$\{\partial P_k\}$ is a Cauchy sequence with respect to $\fn$, so there is a
limit, and the boundary $\partial A$ is defined to be the limit. It is easy
to check that the definition of the boundary does not depend on the choice of
$\{P_k\}$.

For any polyhedral chain $Q=\sum g_i \sigma_i\in \PC_k(\mathbb{R}^n;G)$ and any
polyhedron $\Delta\subseteq \mathbb{R}^n$, we define $Q\mr \Delta=\sum g_i
(\sigma_i \cap \Delta)$. By the Carath\'eodory construction, there is a Radon
measure $\mu_Q$ such that $\mu_Q(\Delta)=\mass(Q \mr \Delta)$. Let $A\in
\FC_k(\mathbb{R}^n;G)$ be any flat chain of finite mass. Then there is a 
sequence of polyhedral chains $\{Q_j\}\subseteq \PC_k(\mathbb{R}^n;G)$ such that
$Q_j\ora{\fn}A$ and $\mass(Q_j)\to \mass(A)$. With $A$ is
associated a Radon measure $\mu_A$, which is defined by the  weak limit of
$\mu_{Q_j}$. Indeed, $\mu_A$ is well defined since that does not depend on the
choice of sequence $\{Q_j\}$, see Section 4 in \cite{Fleming:1966}.
There is also a set function $A\mr E$ with value in $\FC_k(\mathbb{R}^n;G)$, 
which extend $Q\mr \Delta$ to any flat chain $A$ of finite mass and with any 
Borel set $E$, such that $\mu_A(E)= \mass(A\mr E)$. We define the Hausdorff 
size of $A$ to be the value 
\[
	\inf\{\HM^k(E): A\mr E =A, E\text{ is Borel}\}.
\]
Proposition 8.2 in \cite{White:1999:Ann} indicates that flat size coincide
with Hausdorff size in case of finite mass. 

\subsection{Size optimal deformation theorem}
\begin{theorem}\label{thm:pca}
	Let $G$ be any complete normed abelian group. Let $U\subseteq \mathbb{R}^n$ 
	be an open set, $K\subseteq U$ be any compact set. Let $Z\subseteq \mathbb{R}^n$
	be a set such that $0<\HM^d(Z \cap U)<+\infty$ and $\HM^{d+1}(\overline{Z}
	\cap U)=0$. Then there exist constants $c_1=c_1(n,d)\in (0,1)$ and 
	$c_0=c_0(n,d)\geq 1$ such that for any 
	$0< \varepsilon<1/10$ and flat chain $S\in \FC_d(\mathbb{R}^n;G)$ with 
	$0<\mass(S)<+\infty$ and $S\mr Z=S$, we can find a polyhedral complex 
	$\mathcal{K}$ and a Lipschitz mapping $h:I\times \mathbb{R}^n\to 
	\mathbb{R}^n$ satisfying that 
	\begin{enumerate}[label*=(\arabic*),ref=\emph{(\arabic*)}]
		\item \label{pca1} $\mathfrak{r}(\mathcal{K})\leq \varepsilon$,
			$\mathcal{R}(\mathcal{K})\geq c_1$, $|\mathcal{K}|\supseteq K$,
		\item \label{pca2} $h(0,\cdot)=\id$, $\|h(t,\cdot)-\id\|_{\infty}\leq
			\varepsilon$, $h(t,x)=x$ for $x\in \mathbb{R}^n\setminus U$ and $t\in I$,
		\item \label{pca3} $\varphi_{\sharp}S-S=\partial h_{\sharp}(I
			\times S)+h_{\sharp}(I\times\partial S)$, 
		\item \label{pca4} $(\varphi_{\sharp}S) \mr |\mathcal{K}|\in 
			\PC_d(\mathbb{R}^n; G)$,
		\item \label{pca5} $\mass(\varphi_{\sharp}S)\leq c_0 \mass(S)$, 
			$\mass( h_{\sharp}(I\times S)) \leq \varepsilon c_0 \mass(S)$,
		\item \label{pca6} $\mass(h_{\sharp}(I\times\partial S))\leq 
			\varepsilon c_0\mass(\partial S)$ in case $\mass(\partial S)<+\infty$,
		\item \label{pca7} $\size(\varphi_{\sharp}S)\leq \size(S)+\varepsilon$,
		\item \label{pca8} $\size((\varphi_{\sharp}S)\mr \Delta)\leq c_0 
			\size(S\mr (\Delta+\oball(0,\varepsilon)))$ for any $\Delta\in
			\mathcal{K}$,
		\item \label{pca9} $\size((\varphi_{\sharp}S)\mr (U\setminus
			|\mathcal{K}|))\leq c_0 \size(S\mr (U\setminus K))$ in case $S\in
			\RC_d(\mathbb{R}^n;G)$,
	\end{enumerate}
	where $\varphi=h(n-d+2,\cdot)$, $I=[0,n-d+2]$.
\end{theorem}
\begin{proof}
	Put $\mu=\mu_{S}$. Let $X\subseteq Z$ be a Boral set such
	that  $S\mr X =S$ and $\HM^d(X)=\size(S)$. Write $X\cap U=X_{rec}\cup X_{irr}$,
	where $X_{rec}$ is $d$-rectifiable and $X_{irr}$ is purely $d$-unrectifiable.
	Let us apply Theorem \ref{thm:polyapp} with $\ell_0=1$, $F_1\equiv 1$,
	$\lambda_1\equiv 1$ and $\varepsilon_1=\varepsilon_2=\varepsilon/c_0$, and 
	construct the Lipschitz mapping $h$.
	Let the complexes $\mathcal{K}'$ and $\mathcal{K}$, the constant $c_1$, and
	the Lipschitz mapping $\varphi$ be the same as in Theorem \ref{thm:polyapp}.
	But we should carefully select the points $x_{\Delta}\in Y_{\Delta}$ to make
	that hold with \ref{pca5} and \ref{pca6}. Actually, that
	will be done in \eqref{eq:pca200}. 

	Let $\psi, \varphi_i$, $1\leq i\leq n-d+1$, $\mathcal{O}=|\mathcal{K}|$,
	$\mathcal{K}^{(\ell)}=\{\Delta\in \mathcal{K}':
	\dim(\Delta)=\ell,\mathring{\Delta}\cap \partial \mathcal{O}=\emptyset\}$,
	$0\leq \ell\leq n$, be the same as in the proof of Theorem
	\ref{thm:polyapp}.	Define $g:[0,1]\times \mathbb{R}^n\to \mathbb{R}^n$
	by $g(t,\cdot)=(1-t)\id + t\psi$. By (6.3), (6.4), (6.5) and (6.6) in 
	\cite{Fleming:1966}, we get that 
	\[
		\psi_{\sharp}S-S=\partial g_{\sharp}([0,1]\times S)+
		g_{\sharp}([0,1]\times \partial S),
	\]
	\[
		\mass(\psi_{\sharp}S)\leq \Lip(\psi)^d\mass(S)\leq 4^d\mass(S),
	\]
	\[
		\size(\psi_{\sharp}S)\leq \HM^d(\psi_{\sharp}(X))\leq
		\HM^d(X)+\varepsilon=\size(S)+ \varepsilon
	\]
	and 
	\[
		\mass(g_{\sharp}([0,1]\times S))\leq 5^{d+1}\mass(S),
		\mass(g_{\sharp}([0,1]\times \partial S))\leq 4^d \mass(\partial S).
	\]
	Put $S_0=\psi_{\sharp} S$, $\mu_0=\mu_{S_0}$, $S_i=(\varphi_i)_{\sharp}
	S_{i-1}$ and $\mu_i=\mu_{S_i}$, $1\leq i\leq n-d+1$.

	For any compact convex polyhedron $\Delta\subseteq \mathbb{R}^{n}$ with
	$\dim(\Delta)=\ell\geq d+1$, let $\cball_{\Delta}$ and
	$Y_{\Delta}$ be the same as in Lemma \ref{le:projrec}. Assume that
	$\cball_{\Delta}=\cball(x_1,r_0)\subseteq \cball(x_1,2 r_0)\subseteq\Delta
	\subseteq \cball(x_2,R_0)$ and $2r_0/R_0\geq \mathcal{R}(\Delta)/2$.
	Then for any $x\in \mathring{\Delta}$, 
	$\Pi_{\Delta,x}=(p_x\vert_{\partial \Delta})^{-1}\circ p_x$,
	\[
		\Lip(p_x\vert_{\partial \Delta})^{-1}\leq \frac{8R_0^2}{r_0},\
		\|Dp_x(z) \|=\frac{1}{|z-x|}.
	\]
	Thus 
	\[
		\mass \left( (\Pi_{\Delta,x})_{\sharp}(S_{n-\ell}\mr \mathring{\Delta})
		\right)\leq \left(\frac{8R_0}{r_0}\right)^{d}\int_{z\in \mathring{\Delta}}
		\|Dp_x(z)\|^d d\mu_{n-\ell}(z)=\left(\frac{8R_0}{r_0}\right)^{d}
		\int_{z\in \mathring{\Delta}} \frac{d \mu_{n-\ell}(z)}{|z-x|^d}
	\]
	and 
	\[
		\begin{aligned}
			\int_{x\in \cball_{\Delta}}\mass \left( (\Pi_{\Delta,x})_{\sharp}
			(S_{n-\ell}\mr \mathring{\Delta})\right) d\HM^{\ell}(x)&
			\leq \left(\frac{8R_0^2}{r_0}\right)^d \int_{x\in \cball_{\Delta}}
			\int_{z\in \mathring{\Delta}}\frac{d \mu_{n-\ell}(z)}{|z-x|^d}d\HM^{\ell}(x)\\
			&\leq \left(\frac{8R_0^2}{r_0}\right)^dC(\ell)r_0^{\ell-d}
			\mu_{n-\ell}(\mathring{\Delta})\leq C
			\mathcal{R}(\Delta)^{-2d}r_0^{\ell}\mu_{n-\ell}(\mathring{\Delta}),
		\end{aligned}
	\]
	where $C=2^{n}\max\{C(\ell):d+1\leq \ell\leq n\}$. Similar to Lemma
	\ref{le:projrec}, by Chebyshev's inequality, we get that 
	\[
		\HM^{\ell}\left(\left\{x\in \cball_{\Delta}:\mass \left( (\Pi_{\Delta,x})_{\sharp}
			(S_{n-\ell}\mr \mathring{\Delta})\right)\leq C \beta^{-1} \mathcal{R}
			(\Delta)^{-2d}\mu_{n-\ell}(\mathring{\Delta})\right\}
		\right)\geq 
		(1-\beta) \omega_{\ell} r_0^{\ell}.
	\]
	Denote by $Y_{\Delta}'$ the set of points $x\in \cball_{\Delta}$ such that 
	\begin{equation}\label{eq:pca190}
		\mass \left( (\Pi_{\Delta,x})_{\sharp}
			(S_{n-\ell}\mr \mathring{\Delta})\right)\leq 4C \mathcal{R}
			(\Delta)^{-2d}\mu_{n-\ell}(\mathring{\Delta}).
	\end{equation}
	Then we have that 
	\[
		\HM^d(Y_{\Delta}\cap Y_{\Delta}')\geq \omega_{\ell} r_0^{\ell}/2>0.
	\]
	Together with Lemma \ref{le:projirr}, for any $\Delta\in \cup_{d+1\leq
	\ell\leq n} \mathcal{K}^{\ell}$, we can choose point 
	\begin{equation}\label{eq:pca200}
		x_{\Delta}\in Y_{\Delta}\cap Y_{\Delta}' 
	\end{equation}
	such that $\Pi_{\Delta,x_{\Delta}}(X_{irr}\cap \Delta)$ is purely $d$-unrectifiable.
	Thus iequalities \eqref{eq:pca190} and \eqref{eq:projrec1} hold with 
	$x=x_{\Delta}$ and $\beta=1/4$. Similar as in the proof of Theorem
	\ref{thm:polyapp}, write $E\cap U=E_{rec}\cup E_{irr}$, where $E=\psi(X)$.
	 Since $S\mr X=S$, we have that $S_0\mr E=S_0$. Then, by
	\eqref{eq:pca190}, we get that 
	\[
		\mass(S_i)\leq 4C c_1^{-2(n-i)}\mass(S_{i-1}),\ 1\leq i\leq n-d,
	\]
	and $\mass(S_{n-d+1})\leq \mass(S_{n-d})$. Thus $\mass(\varphi_{\sharp} S)
	\leq c_3 \mass(S)$ for some constant $c_3=c_3(n,d)>0$.

	Put $\varphi_0=\psi$, and define the mapping $h:[0,n-d+2]\times \mathbb{R}^n
	\to \mathbb{R}^n$ as follows: 
	\[
		h(0,\cdot)=\id,\ h(i,\cdot)=\varphi_{i-1}\circ\cdots\circ \varphi_0,\
		1\leq i\leq n-d+2;
	\]
	\[
		h(t,\cdot)=(i+1-t)h(i,\cdot)+(t-i)h(i+1,t),\ i<t<i+1,\ 0\leq i\leq n-d+1. 
	\]
	By (6.3) in \cite{Fleming:1966}, we get that \ref{pca3} holds. Indeed, 
	\ref{pca1}, \ref{pca2}, \ref{pca4}, \ref{pca7}, \ref{pca8} and \ref{pca9} 
	directly follow from Theorem \ref{thm:polyapp}, we are going to prove 
	\ref{pca5} and \ref{pca6}.

	For any $1\leq i\leq n-d+2$, define $h_i:[0,1]\times \mathbb{R}^n\to 
	\mathbb{R}^n$ by $h_i(t,\cdot)=(1-t)\id+t \varphi_i$. Put $h_0=g$.
	Then
	\[
		h(t,x)=h_i(t-i,\varphi_{i-1}\circ\cdots\circ \varphi_0(x)),\ (t,x)\in
		[i,i+1]\times \mathbb{R}^n,\ 0\leq i\leq n-d+1,
	\]
	thus
	\[
		\begin{gathered}
			h_{\sharp}(I_i\times S)=(h_i)_{\sharp}\circ (\trans{i},\varphi_{i-1}
			\circ\cdots\circ \varphi_0)_{\sharp}([0,1]\times S),\\
			h_{\sharp}(I_i\times \partial S)=(h_i)_{\sharp}\circ (\trans{i},
			\varphi_{i-1} \circ\cdots\circ \varphi_0)_{\sharp}([0,1]\times
			\partial S),
		\end{gathered}
	\]
	where $I_i=[i,i+1]$, $\trans{i}(t)=t-i$. Hence, by (6.5) in
	\cite{Fleming:1966} and \eqref{eq:pca190}, we have that
	\[
		\mass(h_{\sharp}(I_i\times S))\leq \varepsilon \mass((\varphi_{i-1}
		\circ\cdots\circ \varphi_0)_{\sharp}S)\leq \varepsilon c_3 \mass(S),
	\]
	and 
	\[
		\mass(h_{\sharp}(I\times S))\leq \varepsilon c_0 \mass(S)
	\]
	for some constant $c_0=c_0(n,d)>0$. If $\mass(\partial S)<+\infty$, similarly we have
	that 
	\[
		\mass(h_{\sharp}(I\times \partial S))\leq \varepsilon c_0
		\mass(\partial S).
	\]
	Thus \ref{pca5} and \ref{pca6} hold.

\end{proof}
\subsection{Equal infimum}
In this subsection, we always assume that $G$ is a discrete normed abelian
group. For any compact sets $E,F\subseteq \mathbb{R}^n$, we define their
Hausdorff distance $\HD(E,F)$ by
\[
	\HD(E,F)=\inf\{r\geq 0: E\subseteq F+\cball(0,r), F\subseteq E+\cball(0,r)\}
	.
\]
\begin{lemma}\label{le:MS}
	Assuming $a=\inf\{ \left\vert
	g\right\vert:g\in G\setminus\{ 0 \}\}$, we get that for any $T\in\mathscr{F}
	_{d}(\mathbb{R}^{n};G)$,
	\[
		\mass(T)\geq a\size(T).
	\]
\end{lemma}
\begin{proof}
	For any polyhedral $d$-chain $P\in\PC_{d}(\mathbb{R}^{n};G)$, we
	can write $P=\sum_{i}g_{i}\sigma_{i}$ such that $\sigma_{i}$ are interior
	disjoint, then 
	\[
		\mass(P)=\sum_{i}\left\vert
		g_{i}\right\vert\HM^{d}(\sigma_{i})\geq a\sum_{i}\HM^{d}(\sigma_{i})=a\size(P),
	\]
	hence 
	\[
		\mass(T)=\inf\left\{ \liminf_{j\to \infty}\mass(P_{j}): P_{j}\ora{\fn} T \right\}
		\geq\inf\left\{ \liminf_{j\to \infty}a\cdot\size(P_{j}): P_{j}\ora{\fn} T \right\}
		=a\size(T).
	\]

\end{proof}
\begin{lemma} \label{le:sur}
	Let $S\in \FC_k(\mathbb{R}^n;G)$
	be any flat chain with compact support. Then for any $\varepsilon_2>
	\varepsilon_1>0$ and any sequence of polyhedral chains $P_m\in
	\PC_k(\mathbb{R}^n;G)$ with
	$P_m\ora{\fn}S$, we can find compact set $V$, which is a polyhedron, such
	that  
	\[
		\spt S+\oball(0,\varepsilon_1) \subseteq V \subseteq \spt
		S+\oball(0,\varepsilon_2) \text{ and } P_m\mr V\ora{\fn} S.
	\]
\end{lemma}
\begin{proof}
	By the definition of support, we can find polyhedral chains $\{Q_m\}$ such that
	\[
		Q_m\ora{\fn} S \text{ and }\spt Q_m\subseteq \spt S +\oball(0,\varepsilon_1). 
	\]
	Pass to a subsequence, we may assume that 
	\[
		\sum \fn(P_m-Q_m)<\infty.
	\]
	By Lemma 2.1 in \cite{Fleming:1966}, there exists a compact set $V$, which 
	is a polyhedron, such that 
	\[
		\spt S+\oball(0,\varepsilon_1)\subseteq
		V\subseteq \spt S +\oball(0,\varepsilon_2) \text{ and }
		\sum\fn \big( (P_m-Q_m)\mr (\mathbb{R}^n\setminus V)\big)<\infty.
	\]
	Thus 
	\[
		P_m\mr (\mathbb{R}^n\setminus V)-Q_m\mr (\mathbb{R}^n\setminus V)
		\ora{\fn} 0.
	\]
	However $\spt Q_m\subseteq \intr(V)$, we have that 
	\[
		Q_m\mr (\mathbb{R}^n\setminus V)=0 \text{ and }
		P_m\mr (\mathbb{R}^n\setminus V)\ora{\fn} 0.
	\]
	Thus 
	\[
		P_m\mr V\ora{\fn}S.
	\]

\end{proof}
\begin{lemma}\label{le:aps}
	For any $\varepsilon>0$ and 
	flat chain $R\in \FC_{k}(\mathbb{R}^n;G)$ with compact support, we can find
	a sequence of polyhedral chains $\{P_m\}\subseteq \PC_{k} (\mathbb{R}^n;G)$ such that
	\begin{equation}\label{eq:aps1}
		P_m\ora{\fn} R,\ \spt(P_m)\subseteq \spt(R)+\oball(0,\varepsilon)
		\text{ and } \spt(\partial P_m)\subseteq \spt(\partial
		R)+\oball(0,\varepsilon).
	\end{equation}
\end{lemma}
\begin{proof}
	We take a sequence of polyhedral chains $\{P_m\}$ such that 
	\[
		P_m\ora{\fn}R \text{ and } \spt P_m \subseteq \spt R
		+\oball(0,\varepsilon/10).
	\]
	Then $\partial P_m\ora{\fn} \partial R$. By Lemma \ref{le:sur}, we can
	find a compact polyhedron $V$ such that 
	\[
		\spt \partial R+\oball(0,\varepsilon/10) \subseteq V \subseteq \spt
		\partial R+\oball(0,\varepsilon/5) \text{ and } (\partial P_m)\mr V
		\ora{\fn} \partial R.
	\]
	Thus $\partial ((\partial P_m)\mr V)\ora{\fn} 0 $. By Lemma 7.7 in
	\cite{Fleming:1966}, we can find polyhedral chains $Q_m$ such that 
	\[
		\partial Q_m=\partial ((\partial P_m)\mr V),\ \mass(Q_m)\to 0, \text{ and
		}\spt Q_m\subseteq \spt \partial R+\oball(0,3\varepsilon/10).
	\]
	Thus $(\partial P_m)\mr V -Q_m\ora{\fn} \partial R$, $\partial((\partial P_m)\mr 
	V -Q_m)=0$, and 
	\[
		\spt ((\partial P_m)\mr V -Q_m)\subseteq \spt 
		\partial R+\oball(0,3\varepsilon/10).
	\]
	We get that 
	\[
		\partial P_m - (\partial P_m)\mr V +Q_m\ora{\fn} 0 \text{ and } \partial
		(\partial P_m - (\partial P_m)\mr V +Q_m)=0.
	\]
	Again by Lemma  7.7 in \cite{Fleming:1966}, we can find polyhedral chains
	$T_m$ such that
	\[
		\partial T_m=\partial P_m - (\partial P_m)\mr V +Q_m,\ \mass(T_m)\to 0
		\text{ and }\spt T_m \subseteq \spt(R)+ \oball(0, 2\varepsilon/5).
	\]
	Thus 
	\[
		P_m-T_m\ora{\fn} R \text{ and } \spt (P_m-T_m)\subseteq \spt R +\oball(0,2
		\varepsilon/5),
	\]
	\[
		\partial (P_m-T_m)=(\partial P_m)\mr V -Q_m\ora{\fn} \partial R,
	\]
	\[
		\spt \partial(P_m-T_m)=\spt ((\partial P_m)\mr V -Q_m) \subseteq
		\spt(\partial R) +\oball(0,3 \varepsilon/10).
	\]
	Hence $\{P_m-T_m\}$ is an our desired sequence. 
\end{proof}
\begin{lemma}\label{le:apfc}
	For any flat chain $R\in \FC_{k}(\mathbb{R}^n;G)$ with nonempty compact support,
	we can find a sequence of polyhedral chains $\{P_m\}\subseteq \PC_{k}(\mathbb{R}^n;G)$ such that
	\begin{equation}\label{eq:apfc1}
		P_m\ora{\fn} R,\ \mass(P_m)\to \mass(R),\ \spt(P_m)\ora{\HD}\spt(R)
		\text{ and } \spt(\partial P_m)\ora{\HD}\spt(\partial R).
	\end{equation}
\end{lemma}
\begin{proof}
	First, we will show that, for any sequence of polyhedral chains
	$\{Q_m\}\subseteq \PC_{k}
	(\mathbb{R}^n;G)$ with $Q_m\ora{\fn} R\neq 0$ holds for any $x\in \spt R$, 
	\begin{equation}\label{eq:apfc10}
		\lim_{m\to \infty}\dist(x,\spt Q_m)= 0.
	\end{equation}
	Indeed, if we assume by contradiction that for some $x_0\in \spt R$,
	\begin{equation}\label{eq:11}
		\limsup_{m\to \infty}\dist(x_0,\spt Q_m)=2\delta>0,
	\end{equation}
	then pass to a subsequence, we may assume that $\cball(x_0,\delta)\cap \spt 
	Q_m =\emptyset$ for $m$ large enough. We take $I$ an open interval such that
	$x_0\in I\subseteq \oball(x_0,\delta)$. Then $Q_m\mr I=0$ for $m$ large enough,
	thus $Q_m \mr I\ora{\fn}0$. We put $Q_m'=Q_m \mr (\mathbb{R}^n\setminus I)$, 
	then $Q_m'\in \PC_{k}(\mathbb{R}^n;G)$ and $Q_m'\ora{\fn} R$, thus $R$ is
	supported by $\mathbb{R}^n\setminus I$. By the definition of
	$\spt R$, we have that $I\cap \spt R =\emptyset$. That leads to a contradiction,
	and \eqref{eq:apfc10} holds. 

	Next, we show that for any 
	$\varepsilon>0$ there exists $N_{\varepsilon}>0$ such that
	\begin{equation}\label{eq:apfc20}
		\spt R\subseteq \spt Q_m +\oball(0,\varepsilon),\ \forall m\geq
		N_{\varepsilon}.
	\end{equation}
	Define functions $f_m:\spt R\to \mathbb{R}$ by $f_m(x)=\dist(x,\spt Q_m)$. 
	Then $\Lip(f_m)\leq 1$. Since $\spt R$ is compact, we get that $f_m$
	uniformly converges to $ 0$, and thus \eqref{eq:apfc20} holds.

	By Lemma \ref{le:aps}, there is a sequence of polyhedral chains $\{R_m''\}
	\subseteq \PC_{k}(\mathbb{R}^n;G)$ such that
	\[
		R_m''\ora{\fn} R,\
		\spt(R_m'')\subseteq \spt R+ \oball(0,\varepsilon_m) \text{ and } 
		\spt(\partial R_m'')\subseteq \spt \partial R+ \oball(0,\varepsilon_m), 
	\]
	where $\varepsilon_m\to 0$.
	By Theorem 5.6 in \cite{Fleming:1966}, we can find polyhedral chains 
	$R_m'\in \PC_{k} (\mathbb{R}^n;G)$ such that $R_m'\ora{\fn} R$ and 
	$\mass(R_m')\to \mass(R)$. 
	By Lemma \ref{le:sur}, we can find compact polyhedron $V$ such that 
	\[
		\spt R+\oball(0,\varepsilon)\subseteq V \subseteq \spt
		R+\oball(0,2\varepsilon) \text{ and } R_m'\mr V\ora{\fn} R.
	\]
	So we get that $\partial (R_m'\mr V)\ora{\fn} \partial R$. Thus 
	\[
		\partial (R_m''-(R_m'\mr V))\ora{\fn}0.
	\]
	By Lemma 7.7 in \cite{Fleming:1966}, we can find polyhedral chains $Q_m''\in
	\PC_{k}(\mathbb{R}^n;G)$ such that 
	\[
		\partial Q_m''=\partial (R_m''-R_m'\mr V),\ \mass(Q_m'')\to 0 \text{ and } 
		\spt(Q_m'')\subseteq \spt R+\oball(0,2 \varepsilon).
	\]
	Put $P_m'=Q_m''+R_m'\mr V$. Then
	\[
		P_m'\ora{\fn} R,\ \spt P_m' \subseteq \spt R +\oball(0,2 \varepsilon),\
		\spt \partial P_m' \subseteq \spt \partial R +\oball(0,2 \varepsilon)
	\]
	and 
	\[
		\limsup_{m\to \infty}\mass(P_m')\leq \limsup_{m\to
		\infty}\mass(Q_m'')+\limsup_{m\to \infty}\mass(R_m'\mr V)\leq
		\limsup_{m\to \infty}\mass(R_m')=\mass(R).
	\]
	By lower semi-continuity of mass, we have that $\mass(P_m')\to \mass(R)$.
	Therefore, we can find polyhedral chains $P_m$ such that \eqref{eq:apfc1} hold. 

\end{proof}
\begin{lemma}\label{le:slicing}
	Let $R\in \FC_{k}^c(\mathbb{R}^n;G)$ be any flat chain with $\mass(R)<\infty$.
	Suppose $f:\mathbb{R}^n\to \mathbb{R}$ is a Lipschitz mapping. If $\spt (\partial
	R)\subseteq \{f\leq 0\}$, then we have that for 
	$\HM^1$-a.e. $s\in (0,\infty)$, $ \mass(\partial (R \mr \{f<s\})-
	\partial R)<\infty$, and there is a constant $c=c(k)>0$ such that for any
	$b>a\geq 0$,
\begin{equation}\label{eq:slicing1}
		\int_a^b \mass(\partial (R \mr \{f<s\})-\partial R) ds\leq
		c\Lip(f)\mass(R\mr\{a<f<b\}).
\end{equation}
\end{lemma}
\begin{proof}
	By Lemma \ref{le:apfc}, we can find polyhedral chains $P_m\in
	\PC_{k}(\mathbb{R}^n;G)$ such that
	\[
		P_m\ora{\fn} R,\ \mass(P_m)\to \mass(R)%,\ \spt(P_m)\ora{\HD}\spt(R)
		\text{ and } \spt(\partial P_m)\ora{\HD}\spt(\partial R).
	\]
	Let $\mu_{R}$ be the Radon measure defined by
	$\mu_{R}(X)=\mass(R \mr X)$.
	For any $s\in \mathbb{R}$, we put $X_s=\{f<s\}$, $Y_s=\{f=s\}$,
	and $\Lambda=\{s\in (0,\infty):\mu_R(Y_s)=0\}$. Then
	$\HM^1((0,\infty)\setminus \Lambda)=0$. 
	For any $s\in \mathbb{R}$, we put $B^s=\partial (R \mr X_s)-\partial R$, and
	$B_m^s=\partial (P_m \mr X_s)-(\partial P_m) \mr X_s$. Then for any $s\in
	\Lambda$, we have that $B_m^s\ora{\fn} B^s$, $\spt(B^s)\subseteq Y_s$, and
	$\fn(B^t-B^s)\to 0$ as $t\to s$. We consider the functions
	$g,g_m:[0,\infty)\to [0,\infty]$ given by 
	\[
		g(s)=\mass(B^s) \text{ and }g_m(s)=\mass(B_m^s).
	\]
	Then $g$ is lower semi-continuous and 
	\[
		g(s)\leq \liminf_{m\to \infty}g_m(s), \ \forall s\in \Lambda.
	\]
	By Fatou's Lemma, and the coarea inequality \cite[Theorem
	2.10.25]{Federer:1969}, we get that for any $b>a>0$, if $a,b\not \in
	\Lambda$, then 
	\[
		\int_a^b g(s)ds\leq \liminf_{m\to \infty} \int_a^b g_m(s) ds \leq \liminf_{m\to
		\infty}c_1\mass(P_m \mr \{a<f<b\})=c_1\mass(R\mr \{a<f<b\}),
	\]
	where $c_1=c(k)\cdot \Lip(f)$. Thus we get that \eqref{eq:slicing1}
	holds for any $b>a\geq 0$.
\end{proof}
\begin{lemma}\label{le:fnsm}
	Let $X\subseteq \mathbb{R}^n$ be a compact Lipschitz neighborhood retract.
	Then there exist an $\varepsilon>0$ and a constant $c>0$ such that for any
	$\mathscr{S}\in \{\PC,\LC,\NC,\RC,\FC\}$ and $\sigma\in \mathscr{S}_{k}(X;G)$
	with $\partial\sigma=0$ and $\fn(\sigma)<\varepsilon$, we can find
	$\beta\in \mathscr{S}_{k+1}(X;G)$ such that 
	\begin{equation}\label{eq:fnsm0}
		\partial\beta=\sigma  \text{ and }\mass(\beta)\leq c\fn(\sigma).
	\end{equation}
\end{lemma}
\begin{proof}
	Without loss of generality, we assume $\inf\{|g|:g\in G,g\neq 0\}= 1$.
	Since $X$ is a Lipschitz neighborhood retract, there exists an open set
	$U\supseteq X$ and a Lipschitz mapping $\rho:U\to X$ such that
	$\rho\vert_X=\id_{X}$. We assume $X+\oball(0,\delta)\subseteq U$, chose
	$0<\varepsilon\leq (\delta/c_0)^{k+1}$, and let $c_0>10 n$ be a constant
	which will be
	chosen later. For any $\mathscr{S}\in \{\PC,\LC,\NC,\RC\}$ and any $\sigma\in
	\mathscr{S}_{k}(X;G)$ with $\partial \sigma=0$ and $\fn(\sigma)<\varepsilon$,
	applying the deformation theorem \cite[Deformation theorem]{Fleming:1966}, 
	there exist a constant $c'=c'(n,k)$ and an $\varepsilon'$-cubical grid
	$\chi_{\varepsilon'}$ with $\varepsilon'\leq
	\min\{\delta/(10 n),\fn(\sigma)/\mass(\sigma)\}$, a
	polyhedral chain $Q'\in \PC_k(\mathbb{R}^n;G)$ and $C\in
	\mathscr{S}_{k+1}(\mathbb{R}^n;G)$ such that 
	\begin{itemize}
		\item $\sigma=Q'+\partial C$;
		\item $\mass(C)\leq c' \varepsilon'\mass(\sigma)$;
		\item $\spt(Q')\cup\spt(C)\subseteq \spt(\sigma)+\oball(0,2n \varepsilon')$.
	\end{itemize}
	Since $Q'\in \PC_k(\mathbb{R}^n;G)$ and $\partial Q'=0$, by Lemma 5.4 in
	\cite{Fleming:1966}, we can find constant $c_1>0$ and $P\in \PC_{k+1}(
	\mathbb{R}^n;G)$ such that $\partial P=Q'$ and $\mass(P)\leq c_1\fn(Q')$.

	Let $\Delta$ be an open set which is a union of finite number of open
	polyhedra and such that $X+\oball(0,2n \varepsilon')\subseteq \Delta
	\subseteq X+\oball(0,\delta/10)$. Let 
	function $f:\mathbb{R}^n\to \mathbb{R}$ be defined by 
	\[
		f(x)=\inf\{\|x-y\|_{\infty}:y\in \Delta\}.
	\]
	Then we see that $f$ is Lipschitz and $\Lip(f)=1$
	We take $\varepsilon''=\fn(\sigma)^{1/(k+1)}$, and apply Theorem 5.7 in
	\cite{Fleming:1966} to get that 
	\[
		\int_0^{\varepsilon''}\mass \big(\partial (P\mr
		\{f<s\})-(\partial P)\mr \{f<s\}\big)ds\leq
		\mass\big(P\mr\{0<f<\varepsilon''\}\big),
	\]
	thus there exists $s\in (0,\varepsilon'')$ such that 
	\[
		\mass \big(\partial (P\mr
		\{f<s\})-(\partial P)\mr \{f<s\}\big)\leq
		\frac{1}{\varepsilon''}\mass(P).
	\]
	But $\mass(P)\leq c_1 \fn(Q')\leq c_1(\fn(\sigma)+\fn(\partial C))\leq
	c_1(1+c')\fn(\sigma)$, $\spt(\partial P)=\spt(Q')\subseteq \Delta$, by setting
	$R_s=\partial (P\mr \{f<s\})-Q'$, we have that $R_s\in
	\PC_k(\mathbb{R}^n;G)$ and 
	\[
		\mass(R_s)\leq c_2 \fn(\sigma)^{k/(k+1)}.
	\]
	Applying the isoperimetric inequality \cite[(7.6)]{Fleming:1966} to $R_s$, 
	there is a constant $c''$ and a polyhedral chain $S\in
	\PC_{k+1}(\mathbb{R}^n;G)$ such that 
	\[
		\partial S=R_s,\ \mass(S)\leq c'' \mass(R_s)^{(k+1)/k} \text{ and } 
		\spt(S)\subseteq \spt(R_s)+\oball(0,2n \varepsilon'''),
	\]
	where $\varepsilon'''=(c'')^{1/(k+1)}\mass(R_s)^{1/k}$. Thus $\sigma=Q'+\partial
	C=\partial (P\mr \{f<s\})-\partial S +\partial C$ and 
	\[
		\mass\big(P\mr \{f<s\}-S+C\big)\leq \mass(P)+\mass(S)+\mass(C)\leq
		c_3 \fn(\sigma).	
	\]
	We see that $\spt(P\mr \{f<s\})\subseteq
	\Delta+\oball(0,\varepsilon'')\subseteq X+\oball(0,\delta/5)$ and
	$\spt(C)\subseteq X+\oball(0,\delta/5)$. If we take
	$c_0>10 n$ such that $c_0>10 n (c'')^{1/(k+1)}c_2^{1/k}$, then  
	\[
		2 n \varepsilon'''=2n (c'')^{1/(k+1)}\mass(R_s)^{1/k}\leq
		2n (c'')^{1/(k+1)}c_2^{1/k}\cdot (\delta/c_0)<\delta/5,
	\]
	and $\spt(S)\subseteq \Delta+\oball(0,2n \varepsilon''')\subseteq
	X+\oball(0,\delta)$. We get that 
	\[
		\sigma=\rho_{\sharp}\sigma=\partial \Big(\rho_{\sharp}\big( (P\mr \{f<s\})- S +
		C\big)\Big),
	\]
	$\rho_{\sharp}( (P\mr \{f<s\})- S + C)\in \mathscr{S}_{k+1}(X;G)$, and 
	\[
		\mass\Big(\rho_{\sharp}\big( (P\mr \{f<s\})- S + C\big)\Big)\leq
		\Lip(\rho)^{k+1}c_3\fn(\sigma).
	\]
	Hence, \eqref{eq:fnsm0} holds for $\beta=\rho_{\sharp}( (P\mr \{f<s\})- S +
	C)$ and $c=\Lip(\rho)^{k+1}c_3$.

	If $\mathscr{S}=\FC$, then by the definition of flat norm, there exist 
	$S\in \FC_{k+1}(\mathbb{R}^n;G)$ such that 
	\[
		\mass(\sigma-\partial S)+\mass(S)<2\fn(\sigma).
	\]
	Applying the isoperimetric inequality \cite[(7.6)]{Fleming:1966} to 
	$(\sigma-\partial S)$, there exits $T\in \FC_{k+1}(\mathbb{R}^n;G)$ such that
	$\partial T=\sigma-\partial S$ and 
	\[
		\mass(T)\leq c''\mass((\sigma-\partial S))^{(k+1)/k} \leq
		c''(2\fn(\sigma))^{(k+1)/k}.
	\]
	We take $R=S+T$. Then $\partial R=\sigma$, $\mass(R)\leq c_4\fn(\sigma)<\infty$,
	and by Lemma \ref{le:slicing}, we have that 
	\[
		\int_0^{\varepsilon''}\mass \big(\partial (R\mr
		\{f<s\})-\sigma\big)ds\leq c(k)
		\mass\big(R\mr\{0<f<\varepsilon''\}\big),
	\]
	thus there exists $s\in (0,\varepsilon'')$ such that 
	\[
		\mass\big(\partial (R\mr \{f<s\})-\sigma\big)\leq
		\frac{c(k)}{\varepsilon''}\mass(R)\leq c(k)c_4 \fn(\sigma)^{k/(k+1)}.
	\]
	Applying the isoperimetric inequality \cite[(7.6)]{Fleming:1966} to
	$\sigma-\partial (R\mr \{f<s\})$, there is a flat chain $A\in 
	\FC_{k+1}(\mathbb{R}^n;G)$ such that $\partial A=\sigma-\partial (R\mr
	\{f<s\})$, $\spt(A)\subseteq \{f<s\}+\oball(0,2n \varepsilon''')\subseteq U$, and 
	\[
		\mass(A)\leq c''\mass \big(\sigma-\partial (R\mr
		\{f<s\})\big)^{(k+1)/k}\leq c''(c(k)c_4)^{(k+1)/k}\fn(\sigma)=c_5\fn(\sigma).
	\]
	Thus $\sigma=\partial (A+R\mr \{f<s\})$, and
	$\sigma=\rho_{\sharp}\sigma=\partial \rho_{\sharp}(A+R\mr \{f<s\})$, 
	\[
		\mass(\rho_{\sharp}(A+R\mr \{f<s\})\leq \Lip(\rho)^{k+1}\mass(A+R\mr
		\{f<s\})\leq \Lip(\rho)^{k+1}(c_4+c_5)\fn(\sigma).
	\]
	Hence, \eqref{eq:fnsm0} holds for $\beta=\rho_{\sharp}(A+ R\mr \{f<s\})$ and
	$c=\Lip(\rho)^{k+1}(c_4+c_5)$.

\end{proof}

\begin{lemma}
	\label{le:appl}
	If $X\subseteq\mathbb{R}^{n}$ is a compact Lipschitz neighborhood retract,
	$T\in\mathscr{F}_{k}(X;G)$ is flat $k$-chain such that  $\partial T\in
	\LC_{k-1}(X;G)$, then we can find a sequence of Lipschitz chains
	$\{T_{m}\}_{m\geq 1}$ such that $\partial T_{m}=\partial T$ and 
	$\fn(T-T_{m})\to 0$.
\end{lemma}
\begin{proof}
	Since $X$ is a Lipschitz neighborhood retract, we can find an open set
	$U\supseteq X$ and a Lipschitz mapping $\rho:U\to X$ such that
	$\rho\vert_X=\id_X$. Applying Lemma \ref{le:apfc} to $T\in
	\mathscr{F}_{k}(X;G)$, we can find polyhedral chains $P_m\in
	\PC_{k}(\mathbb{R}^n;G)$ such that $P_m\ora{\fn} T$ and $\spt(P_m)\ora{\HD}\spt(T)
	$. We put $T_m'=\rho_{\sharp}(P_m)$, then $T_m'$ are Lipschitz chains in $X$ and
	satisfying that 
	\[
		\fn(T-T'_{m})=\fn(\rho_{\sharp}(T-P_m))\leq \Lip(\rho)^{k+1}\fn(T-P_m)\to 0.
	\]
	Then 
	\[
		\fn(\partial T-\partial T'_{m})\leq \fn(T-T'_{m})\to 0.
	\]
	Applying Lemma~\ref{le:fnsm}, for $m$ large enough, we can find $S_{m}\in
	\LC_{k}(X;G)$ such that $\partial T-\partial T'_{m}=\partial S_{m}$ and
	$\mass(S_{m})\to 0$. We now take $T_{m}=T'_{m}+S_{m}$, then we have that $\partial
	T_{m}=\partial T$, $T_{m}\in \LC_{k}(X;G)$ and 
	\[
		\fn(T-T_{m})=\fn(T-T'_{m}-S_{m})\leq \fn(T-T'_{m})+\fn(S_{m})\to 0.
	\]
\end{proof}
\begin{lemma}\label{le:cnon}
	For any $\mathscr{S}\in \{\NC,\RC,\FC\}$ and 
	$T\in\mathscr{S}^c_{d-1}(\mathbb{R}^{n};G)$ with $\partial T=0$, we have that
	\[
		\inf\left\{\size(S):\partial S=T,\ S\in\mathscr{S}_{d}(\mathbb{R}^{n};G) \right\}
		=\inf\left\{ \size(S): \partial S=T,\ S\in\mathscr{S}^c_{d}(\mathbb{R}^{n};G) \right\}
	\]
\end{lemma}
\begin{proof}
	We claim that for any Lipschitz mapping $\varphi:\mathbb{R}^n\to \mathbb{R}^n$
	and $R\in \FC_k(\mathbb{R}^n;G)$, we have that 
	\[
		\size(\varphi_{\sharp}R)\leq \Lip(\varphi)^k\size(R).
	\]
	Indeed, we take polyhedral chains $\{P_m\}$ such that $P_m\ora{\fn}R$ and
	$\size(P_m)\to \size(R)$. Then $\varphi_{\sharp}P_m\ora{\fn} \varphi_{\sharp}
	R$, and 
	\[
		\size(\varphi_{\sharp}R)\leq \liminf_{m\to \infty}\size(\varphi_{\sharp}P_m)\leq
		\liminf_{m\to \infty}\Lip(\varphi)^k \size(P_m)=\Lip(\varphi)^k\size(R).
	\]

	Let us tend to prove the lemma. We assume $\spt T \subseteq \cball(0,R_0)$
	for some $R_0>0$. We consider the mapping $\varphi:\mathbb{R}^n\to
	\mathbb{R}^n$ given by 
	\[
		\varphi(x)=\begin{cases}
			x,&x\in \cball(0,R_0),\\
			\frac{R_0}{|x|}x,& |x|>R_0.
		\end{cases}
	\]
	We see that $\varphi$ is Lipschitz with $\Lip(\varphi)=1$. For any $S\in
	\mathscr{S}(\mathbb{R}^n;G)$ with $\partial S=T$, we have that $\partial
	\varphi_{\sharp}S=\varphi_{\sharp}(\partial S)=\varphi_{\sharp}T=T$ and
	$\size(\varphi_{\sharp}S)\leq \size(S)$. Thus 
	\[
		\inf\{\size(S):\partial S=T,\ S\in\mathscr{S}_{d}(\mathbb{R}^{n};G) \}
		\geq\inf\{ \size(S): \partial S=T,\
		S\in\mathscr{S}^c_{d}(\mathbb{R}^{n};G)\}.
	\]
	The reverse inequality is clear, so we get that the equality.
\end{proof}
\begin{lemma}
	\label{le:LF}
	For any $\mathscr{S}\in \{\LC,\NC,\RC\}$ and 
	$T\in\mathscr{S}_{d-1}(\mathbb{R}^{n};G)$, if
	$\partial T=0$ and $\spt(T)$ is compact, then we have that
	\[
		\inf\left\{\size(S):\partial S=T,\ S\in\mathscr{S}_{d}(\mathbb{R}^{n};G) \right\}
		=\inf\left\{ \size(S): \partial S=T,\ S\in\mathscr{F}_{d}(\mathbb{R}^{n};G)
		\right\}.
	\]
\end{lemma}
\begin{proof}
	We let $S$ be any flat $d$-chain
	of compact support with $\partial S=T$. Suppose that $\spt(S)$ is contained 
	in a large ball $\oball(0,r)$. Then we can find a sequence of chains
	$\{S_{m}'\}_{m\geq 1}\subseteq \mathscr{S}_{d}(\mathbb{R}^n;G)$ with $ \spt(S_{m}')\subseteq
	\cball(0,r+1) $ such that $\fn(S_{m}'-S)\to 0$ and 
	\[
		\size(S)=\lim_{m\to\infty}\size(S_{m}').
	\]
	Since $ \fn(\partial S_{m}'-\partial S)\leq\fn(S_{m}'-S)\to 0$, and 
	$\partial S_{m}'-\partial S=\partial S_m'-T\in
	\mathscr{S}_{d-1}(\mathbb{R}^n;G)$, by Lemma~\ref{le:fnsm}, we can find chains
	$ W_{m}\in \mathscr{S}_{d}(\cball(0,r+1);G)$ such that $\partial S_{m}'-\partial S=\partial W_m$ and 
	\[
		\mass(W_{m})\leq c\fn(\partial S_{m}'-\partial S)\leq
		c\fn(S_{m}'-S).
	\]
	Since $G$ is a discrete abelian group, we set $a:=\inf\{\left\vert
	g\right\vert: g\in G\setminus \{0\}\}>0$, then by lemma~\ref{le:MS}, we have 
	\[
		\mass(W_{m})\geq a\size(W_{m}),
	\]
	thus $\size(W_{m})\to 0$. We now put $S_{m}=S_{m}'-W_{m}$. Then $\{S_{m}\}_{m}$
	is a sequence of chains in $\mathscr{S}_{d}(\mathbb{R}^n;G)$ such
	that $\partial(S_m)=T$ and	$\fn(S_{m}-S)\to 0$, thus
	\begin{align*}
		\size(S)&\leq \liminf_{m\to\infty}\size(S_{m})=
		\liminf_{m\to\infty}\size(S_{m}'-W_{m})\\
		&\leq	\liminf_{m\to\infty}(\size(S_{m}')+\size(W_{m}))
		=\liminf_{m\to\infty}\size(S_{m}') =\size(S),
	\end{align*}
	so we have
	\[
		S_{m}\ora{\fn} S,\ \partial S_{m}=\partial S\text{ and
		}\size(S)=\lim_{m\to\infty}\size(S_{m}).
	\]
	And we get that 
	\[
		\inf\left\{	\size(S):\partial S=T,S\in \mathscr{F}_{d}(\mathbb{R}^{n};G)\right\}
		=\inf\left\{	\size(S):\partial S=T, S\in \mathscr{S}_{d}(\mathbb{R}^{n};G)
		\right\}.
	\]
\end{proof}

\begin{lemma}
	\label{le:sizediff}
	 Suppose that $\mathscr{S}\in
	\{\LC,\NC,\RC,\FC\}$ and $T,T'\in \mathscr{S}_{d-1}(\mathbb{R}^n;G)$. If 
	$R\in \mathscr{S}_d(\mathbb{R}^n;G)$ satisfying that $T-T'=\partial R$, 
	then we have that 
	\[
		|\inf\{\size(S):S\in \mathscr{S}_d(\mathbb{R}^n,G),\partial
		S=T\}-\inf\{\size(S):S\in \mathscr{S}_d(\mathbb{R}^n,G),\partial S=T'\}|\leq \size(R).
	\]
\end{lemma}
\begin{proof}
	For convenience, we put 
	\[
		\alpha=\inf\{\size(S):S\in \mathscr{S}_d
		(\mathbb{R}^n,G),\partial S=T\},\ 
		\alpha'=\inf\{\size(S):S\in
		\mathscr{S}_d(\mathbb{R}^n,G),\partial S=T'\}.
	\]
	We take two sequences $\{S_m\},\{S_m'\}\subseteq \mathscr{S}_d(\mathbb{R}^n;G)$
	such that 
	\[
		\partial S_m=T,\ \partial S_m'=T',\ \size(S_m)\to \alpha \text{ and }
		\size(S_m')\to \alpha'.
	\] 
	Then $\partial (S_m-R)=T'$ and $\partial
	(S_m'+R)=T$, thus
	\[
		\alpha \leq \size(S_m'+R)\leq \size(S_m')+\size(R),\ \alpha'\leq \size(S_m-R)\leq 
		\size(S_m)+\size(R),
	\]
	and 
	\[
		\alpha\leq \alpha'+\size(R), \ \alpha'\leq \alpha+\size(R).
	\]
\end{proof}
\begin{lemma}\label{le:fap}
	Let $A\in \NC_d(\mathbb{R}^n;G)$ be any normal chain with compact
	support. Suppose that $\HM^{d+1}(\spt A)=0$. Then for any $\varepsilon>0$, 
	we can find a sequence of polyhedral chains
	$\{P_m\}\subseteq \PC_d(\mathbb{R}^n;G)$, a sequence of normal chains
	$\{R_m\}\subseteq \NC_d(\mathbb{R}^n;G)$ and a constant $C=C(n,d)>0$ such that
	, 
	\begin{equation}\label{eq:fap1}
		P_m\ora{\fn} A,\ \size(P_m)\to \size(A),
	\end{equation}
\begin{equation}\label{eq:fap2}
		\partial(P_m+R_m)=\partial A,\ \spt(R_m)\cup\spt (\partial P_m)
		\subseteq \spt \partial A+\oball(0,\varepsilon),
	\end{equation}
	and for any compact set $K\subseteq \mathbb{R}^n$ ,
	\begin{equation}\label{eq:fap3}
		\limsup_{m\to \infty}\HM^d((\spt P_m)\cap K)\leq C\size(A\mr
		(K+\oball(0,\varepsilon)).
	\end{equation}
	Moreover, if $A\in \LC_d(\mathbb{R}^n;G)$, then $R_m\in
	\LC_d(\mathbb{R}^n;G)$.
\end{lemma}
\begin{proof}
	Take Borel set $X\subseteq \spt A$ such that $A\mr X =A$ and
	$\HM^d(X)=\size(A)$. Assume that $\spt A\subseteq \cball(0,r)$. Applying
	Theorem \ref{thm:pca} with $S=A$ and $K=\cball(0,r)$, for any $\delta>0$, 
	there is a Lipschitz mapping $h:[0,1]\times \mathbb{R}^n\to \mathbb{R}^n$ 
	such that $\|h(t,\cdot)-\id\|_{\infty}\leq \delta$, $h(1,\cdot)_{\sharp}
	A\in \PC_d(\mathbb{R}^n;G)$,
	\[
		A= h(1,\cdot)_{\sharp} A- h_{\sharp}([0,1]\times \partial A)-
		\partial h_{\sharp}([0,1]\times A),
	\]
	and 
	\[
		\size(h(1,\cdot)_{\sharp} A)\leq \size(A)+\delta.
	\]
	For $\delta=1/2^m$, we take $P_m=h(1,\cdot)_{\sharp} A$ and $R_m=-h_{\sharp}
	([0,1]\times \partial A)$, then
	\eqref{eq:fap1}, \eqref{eq:fap2} and \eqref{eq:fap3} hold.

\end{proof}

\begin{proposition}
	\label{prop:flatsize}
	Let $B_0\subseteq \mathbb{R}^{n}$ be a compact Lipschitz neighborhood retract 
	with $\HM^d(B_0)=0$.	For any $\mathscr{S}\in \{\LC,\NC,\RC,\FC\}$ 
	and $T\in \mathscr{S}_{d-1}(B_0;G)$ with $\partial T=0$, we have that
	\begin{equation}\label{eq:fs1}
		\inf\{\size(S): \partial S=T,S\in\mathscr{S}_{d}(\mathbb{R}^{n};G)\}
		=\inf\{\HM^{d}(E): E\in \mathscr{C}_{\mathscr{S}}(B_0,G,T)\}.
	\end{equation}
\end{proposition}

\begin{proof}
	For any $E\in \mathscr{C}_{\mathscr{S}}(B_0,G,T)$, there exists $S\in
	\mathscr{S}_{d}(E;G)$ such that $\partial S=T$, thus 
	\[
		\size(S)\leq \HM^{d}(\spt(S))\leq\HM^{d}(E),
	\]
	and 
	\[
		\inf\left\{	\size(S): \partial S=T,S\in
		\mathscr{S}_{d}(\mathbb{R}^{n};G)\right\}\leq \inf\{\HM^{d}(E): E\in
		\mathscr{C}_{\mathscr{S}}(B_0,G,T)\}.
	\]

	We now turn to prove the reverse inequality, that is,  
	for any $S\in \mathscr{S}_{d}(\mathbb{R}^{n};G)$ with
	$\partial S=T$, 
	\begin{equation}\label{eq:fs50}
		\size(S)\geq\inf\{\HM^{d}(E): E\in \mathscr{C}_{\mathscr{S}}(B_0,G,T)\}.
	\end{equation}

	Since $B_0$ is a compact Lipschitz neighborhood retract, there is an open set
	$U\supseteq B_0$ and a Lipschitz mapping $\rho:U\to B_0$ such that
	$\rho\vert_{B_0}=\id_{B_0}$. For any $\varepsilon>0$ with $B_0+\oball(0,10
	\varepsilon)\subseteq U$, we let
	$\psi_{\varepsilon}:\mathbb{R}^n\to \mathbb{R}^n$ be a Lipschitz extension
	of mapping $\varphi_{\varepsilon}$ defined by 
	\[
		\varphi_{\varepsilon}(x)=
		\begin{cases}
			x,&x\in \mathbb{R}^n\setminus B_0+\oball(0,2 \varepsilon),\\
			\rho(x),&x\in B_0+\cball(0,\varepsilon).
		\end{cases}
	\]
	Since $\Lip(\varphi_{\varepsilon})\leq 2+\Lip(\rho)$, we can assume 
	$\Lip(\psi_{\varepsilon})=\Lip(\varphi_{\varepsilon})\leq 2+\Lip(\rho)$.

	If$\mathscr{S}=\LC$, $T\in \LC_{d-1}(\mathbb{R}^n;G)$, $\spt(T)\subseteq
	B_0$, $\partial T=0$. For any $S\in \LC_d(\mathbb{R}^n;G)$ with $\partial
	S=T$, by Lemma \ref{le:fap}, we can find a sequence $\{P_m\}\subseteq
	\PC_d(\mathbb{R}^n;G)$ and a sequence $\{R_m\}\subseteq \LC_d(\mathbb{R}^n;G)$
	such that $\partial (P_m+R_m)=T$, $\spt(\partial P_m)\cup \spt (R_m)
	\subseteq B_0+\oball(0,\varepsilon)$, $P_m\to S$, $\size(P_m)\to\size(S)$,
	and for any compact set $K$, 
	\[
		\limsup_{m\to \infty}\HM^d(\spt(P_m)\cap K)\leq C \size(S\mr (K+\oball(0,
		\varepsilon))).
	\]

	We put $E_m=B_0\cup \spt (P_m) \cup \spt (R_m)$. Since $\partial
	(P_m+R_m)= T$, we get that $E_m$ spans $T$, i.e. 
	$E_m\in \mathscr{C}_{\LC}(B_0,G,T)$. Thus
	$\psi_{\varepsilon}(E_m)\in \mathscr{C}_{\LC}(B_0,G,T)$. Since
	$\spt(R_m)\subseteq B_0+\oball(0,\varepsilon)$, we get that
	$\psi_{\varepsilon}(\spt R_m)\subseteq B_0$ and
	\[
		\psi_{\varepsilon}(E_m)=B_0\cup \psi_{\varepsilon}(\spt P_m).
	\]
	Setting $U_{\varepsilon}=B_0+\oball(0,2 \varepsilon)$, we get that
	$\psi_{\varepsilon}(\spt(P_m)\setminus U_{\varepsilon})=\spt(P_m)\setminus
	U_{\varepsilon}$ and 
	\[
		\psi_{\varepsilon}(E_m)= B_0\cup \psi_{\varepsilon}(\spt(P_m)\cap
		U_{\varepsilon})\cup (\spt(P_m)\setminus U_{\varepsilon}).
	\]
	Thus 
	\[
		\begin{aligned}
			\HM^d(\psi_{\varepsilon}(E_m))&\leq \HM^d(\spt(P_m)\setminus
			U_{\varepsilon})+\HM^d(\psi_{\varepsilon}(\spt(P_m)\cap
			U_{\varepsilon}))\\
			&\leq
			\HM^d(\spt P_m)+\Lip(\psi_{\varepsilon})^d\HM^d(\spt(P_m)\cap
			U_{\varepsilon})
		\end{aligned}
	\]
	By Lemma \ref{le:fap}, we get that 
	\[
		\limsup_{m\to\infty}\HM^d(\spt(P_m)\cap U_{\varepsilon})\leq
		C \size(S \mr (B_0+\oball(0,3 \varepsilon)).
	\]
	and 
	\[
		\limsup_{m\to \infty}\HM^d(\psi_{\varepsilon}(E_m))\leq
		\size(S)+C\cdot (2+\Lip(\rho))^{d}\size(S\mr (B_0+\oball(0,3\varepsilon))).
	\]
	Let $\nu$ be the Radon measure given by $\nu(U)=\size(S\mr U)$. Then we see
	that $\nu(B_0)=0$. Since $\HM^d(B_0)=0$, we get that 
	\[
		\limsup_{m\to \infty}\HM^d(\psi_{\varepsilon}(E_m))\leq
		\size(S)+C'\lim_{\varepsilon\to 0}
		\nu(B_0+\cball(0,\varepsilon))=\size(S)+C'\nu(B_0)=\size(S),
	\]
	and 
	\[
		\inf\{\HM^d(E):E\in \mathscr{C}_{\LC}(B_0,G,T)\}\leq \size(S).
	\]
	Thus the euqality \eqref{eq:fs1} holds for $\mathscr{S}=\LC$.

	If $\mathscr{S}\in \{\NC,\RC,\FC\}$, by Lemma \ref{le:appl}, we can find
	$\{T_m\}\subseteq \LC_{d-1}(B_0;G)$ such that $\partial T_m=0$ and
	$\fn(T-T_m)\to 0$. By Lemma \ref{le:fnsm}, we can find $\{R_m\}\subseteq
	\mathscr{S}_d(B_0;G)$ such that $T-T_m=\partial R_m$ and $\mass(R_m)\leq
	c\fn(T- T_m)\to 0$. By Lemma \ref{le:MS}, we get that $\size(R_m)\to 0$.
	By Lemma \ref{le:LF}, we get that 
	\[
		\begin{aligned}
			\inf\{\size(S):\partial S=T_m, S\in \mathscr{S}_d(\mathbb{R}^n;G)\}
			&=\inf\{\size(S):\partial S=T_m, S\in \LC_d(\mathbb{R}^n;G)\}\\
			&\geq \inf\{\HM^d(E):E\in \mathscr{C}_{\LC}(B_0,T_m,G)\}\\
			&\geq \inf\{\HM^d(E):E\in \mathscr{C}_{\mathscr{S}}(B_0,T_m,G)\}\\
			&=\inf\{\HM^d(E):E\in \mathscr{C}_{\mathscr{S}}(B_0,G,T)\}.
		\end{aligned}
	\]
	By Lemma \ref{le:sizediff}, we have that 
	\[
		\inf\{\size(S):\partial S=T_m, S\in \mathscr{S}_d(\mathbb{R}^n;G)\}\leq
		\inf\{\size(S):\partial S=T, S\in
		\mathscr{S}_d(\mathbb{R}^n;G)\}+\size(R_m).
	\]
	We get that 
	\[
		\inf\{\size(S):\partial S=T, S\in \mathscr{S}_d(\mathbb{R}^n;G)\}
		\geq \inf\{\HM^d(E):E\in \mathscr{C}_{\mathscr{S}}(B_0,G,T)\}.
	\]

\end{proof}
\begin{proposition}
	\label{prop:fssb}
	Let $B_0\subseteq \mathbb{R}^{n}$ be a compact Lipschitz neighborhood retract.
	Then for any $\sigma\in \Hom_{d-1}^{\NC}(B_0;G)$, we have that 
	\begin{equation}\label{eq:fssb}
		\begin{aligned}
			&\inf\{\size(S): S\in\mathcal{Z}_{d}^{\NC}(\mathbb{R}^{n},B_0;G),
			[\partial S]=\sigma\}\\
			&=\inf\{\size(S \mr (\mathbb{R}^n\setminus B_0)): S\in\mathcal{Z}_{d}^{\NC}(\mathbb{R}^{n},B_0;G),
			[\partial S]=\sigma\}\\
			&=\inf\{\HM^{d}(E\setminus B_0): E\in \mathscr{C}_{\NC}(B_0,G,\sigma)\}.
		\end{aligned}
	\end{equation}
\end{proposition}
\begin{proof}
	Put $U=\mathbb{R}^n\setminus B_0$. Let $U_0\supseteq B_0$ be an open set,
	$\rho:U_0\to B_0$ be a Lipschitz retraction. Suppose that
	$B_0+\cball(0,\delta_0)\subseteq U_0$, $\delta_0>0$. Let $f:\mathbb{R}^n
	\to \mathbb{R}$ be the Lipschitz function defined by $f(x)=\dist(x,B_0)$.

	For any $E\in \mathscr{C}_{\NC}(B_0,G,\sigma)$, there exists $S\in
	\NC_{d}(E;G)$ such that $[\partial S]=\sigma$, thus 
	\[
		\size(S\mr U)\leq \HM^{d}(E\setminus B_0),
	\]
	and 
	\[
		\inf\left\{	\size(S\mr U): S\in \mathcal{Z}_{d}^{\NC}
		(\mathbb{R}^{n},B_0;G), [\partial S]=\sigma\right\}\leq 
		\inf\{\HM^{d}(E\setminus B_0): E\in \mathscr{C}_{\NC}(B_0,G,\sigma)\}.
	\]

	For any $0< \varepsilon< \delta_0/2$, we let
	$\rho_{\varepsilon}:\mathbb{R}^n\to \mathbb{R}^n$ be a Lipschitz extension
	of mapping $\rho_{\varepsilon}'$ which is defined by 
	\[
		\rho_{\varepsilon}'(x)=\begin{cases}
			x,&f(x)\geq 2 \varepsilon,\\
			\rho(x),& f(x)\leq \varepsilon.
		\end{cases}
	\]
	Indeed, we can suppose that $\Lip(\rho_{\varepsilon})=
	\Lip(\rho_{\varepsilon}')\leq 2+
	\Lip(\rho)$. For any $R\in \NC_d(\mathbb{R}^n;G)$ with 
	$\spt \partial R\subseteq B_0$ and $[\partial R]=\sigma$, by Theorem 5.7 in
	\cite{Fleming:1966}, we get that for $\HM^1$-a.e. $s\in (0,\varepsilon)$,
	\[
		R\mr\{f>s\}\in \NC_{d}(\mathbb{R}^n;G) \text{ and }
		\spt \partial (R\mr\{f>s\})\subseteq \{f=s\}.
	\]
	Since $\spt (R-R\mr\{f>s\})\subseteq \{f\leq s\}$, we get that 
	\[
		(\rho_{\varepsilon})_{\sharp} R- (\rho_{\varepsilon})_{\sharp} (R \mr
		\{f>s\})\in \NC_{d}(B_0;G).
	\]
	Thus
	\[
		(\rho_{\varepsilon})_{\sharp}( R \mr \{f>s\})\in
		\mathcal{Z}_{d}^{\NC}(\mathbb{R}^{n},B_0;G) \text{ and }
		[\partial (\rho_{\varepsilon})_{\sharp} (R \mr \{f>s\})]=[\partial
		R]=\sigma.
	\]
	We see that
	\[
		\size((\rho_{\varepsilon})_{\sharp}( R \mr \{f>s\}))\leq
		\Lip(\rho_{\varepsilon})^d \size(R\mr \{s\leq f\leq 2 \varepsilon\})+
		\size(R\mr \{f> 2 \varepsilon\}),
	\]
	and 
	\[
		\begin{aligned}
			&\inf\{\size(S):
			S\in\mathcal{Z}_{d}^{\NC}(\mathbb{R}^{n},B_0;G), [\partial
			S]=\sigma\}\\
			&\leq (2+\Lip(\rho))^d\size(R\mr \{s\leq f\leq 2 \varepsilon\})+
			\size(R\mr \{f> 2 \varepsilon\}),
		\end{aligned}
	\]
	let $\varepsilon$ tend to 0, we will get that the first equality in
	\eqref{eq:fssb} holds. Therefore
	\[
		\inf\{\size(S): S\in\mathcal{Z}_{d}^{\NC}(\mathbb{R}^{n},B_0;G), 
		[\partial S]=\sigma\}\leq \inf\{\HM^{d}(E\setminus B_0): E\in
		\mathscr{C}_{\NC}(B_0,G,\sigma)\}.
	\]
	For any $S\in\mathcal{Z}_{d}^{\NC}(\mathbb{R}^{n},B_0;G)$ with
	$[\partial S]=\sigma$, by Theorem \ref{thm:pca}, there is a Lipschitz
	mapping $h$ such that $h(t,x)=x$ for $x\in B_0$, $0\leq t\leq n-d+2$,
	$\varphi_{\sharp}S=S+\partial h_{\sharp}(I\times S)+h_{\sharp}(I\times
	\partial S)$, where $\varphi=h(n-d+2,\cdot)$. Since $\spt \partial S\subseteq
	B_0$, we get that $\partial \varphi_{\sharp}S =\varphi_{\sharp}(\partial S)
	=\partial S$, thus $\spt \varphi_{\sharp}S\in \mathscr{C}_{\NC}(B_0,G,\sigma)$.
	Thus $\rho_{\varepsilon}(\spt \varphi_{\sharp}S)\in
	\mathscr{C}_{\NC}(B_0,G,\sigma)$, and 
	\[
		\begin{aligned}
			\HM^d(\rho_{\varepsilon}(\spt \varphi_{\sharp}S)\setminus B_0)&\leq
			\Lip(\rho_{\varepsilon})^d\HM^d((\spt \varphi_{\sharp}S) \cap B_{2
			\varepsilon}\setminus B_{\varepsilon})+\HM^d((\spt
			\varphi_{\sharp}S)\setminus B_{2 \varepsilon})\\
			&\leq (2+\Lip(\rho))^dc_0\size(S\mr (B_{3 \varepsilon}\setminus
			B_0))+\size(S\mr U)+\varepsilon'.
		\end{aligned}
	\]
	Let $\varepsilon$ tend to 0, we will get that 
	\[
		\inf\{\HM^{d}(E\setminus B_0): E\in
		\mathscr{C}_{\NC}(B_0,G,\sigma)\}\leq \size(S\mr U)+ \varepsilon'.
	\]
	Thus
	\[
		\inf\{\HM^{d}(E\setminus B_0): E\in
		\mathscr{C}_{\NC}(B_0,G,\sigma)\}\leq \inf\{\size(S\mr U): S\in\mathcal{Z}_{d}^{\NC}(\mathbb{R}^{n},B_0;G), 
		[\partial S]=\sigma\}.
	\]

\end{proof}
\section{Currents}
Let $\PC_d(\mathbb{R}^n)$ be the group of integral polyhedral chains in
$\mathbb{R}^n$ defined in \cite[4.1.22]{Federer:1969}. It is clear that 
$(\PC_d(\mathbb{R}^n),\mass)$ and $(\PC_d(\mathbb{R}^n;\mathbb{Z}),\mass)$ are
natural isometrically isomorphic, thus by definition, $(\LC_d(\mathbb{R}^n),\mass)$ 
and $(\LC_d(\mathbb{R}^n;\mathbb{Z}),\mass)$ are natural isometrically
isomorphic. And the natural isometric
isomorphism commute with the boundary operator $\partial$ and any induced
homomorphism $f_{\sharp}$ for Lipschitz mappings $f$.
We let 
\[
	\FC_d^{loc}(\mathbb{R}^n)=\{R+\partial S:R\in \RC_d^{loc}(\mathbb{R}^n),
	S\in \RC_{d+1}^{loc}(\mathbb{R}^n)\},
\]
and call $d$-dimensional locally integral flat chains the elements of
$\FC_d^{loc}(\mathbb{R}^n)$. A $d$-dimensional integral flat chain is a
$d$-dimensional locally integral flat chain with compact support, and denote
by $\FC_d(\mathbb{R}^n)$ the collection of all integral flat chains.
For any $T\in \IC_d^{loc}(\mathbb{R}^n)$, we set 
$
\mathsf{N}(T)=\mass(T)+\mass(\partial T),
$
and for any $T\in \FC_d^{loc}(\mathbb{R}^n)$, we set 
\[
	\fn(T)=\inf\{\mass(R)+\mass(S):T=R+\partial S, R\in \RC_d^{loc}(\mathbb{R}^n),
	S\in \RC_{d+1}^{loc}(\mathbb{R}^n) \}.
\]
\begin{lemma}\label{le:isoR}
	$(\RC_d^{loc}(\mathbb{R}^n),\mass)$ and $(\RC_d(\mathbb{R}^n;\mathbb{Z}),\mass)$ are
	natural  isometrically isomorphic.
\end{lemma}
\begin{proof}
	We denote by $\ell_{\ast}:\LC_{\ast}(\mathbb{R}^n)\to \LC_{\ast}
	(\mathbb{R}^n;\mathbb{Z})$ the natural isometric isomorphism.
	By Theorem \cite[4.1.28]{Federer:1969}, for any $T\in \RC_d^{loc}(\mathbb{R}^n)$, and any
	$\varepsilon>0$, we can find Lipschitz mapping $f:\mathbb{R}^n\to
	\mathbb{R}^n$ and $P\in \PC_d(\mathbb{R}^n)$ such that 
	\[
		\mass(T-f_{\sharp}P)<\varepsilon,
	\]
	thus	we can find a sequence $\{T_m\}\subseteq \LC_d(\mathbb{R}^n)$ such
	that $T_m\ora{\mass} T$. Then $\{\ell_d(T_m)\}\subseteq
	\LC_d(\mathbb{R}^n;\mathbb{Z})$ is a Cauchy sequence for mass, so it is also
	a Cauchy sequence for flat norm, and it converges to a flat chain, saying
	$\ell_d^{\RC}(T)$. Thus 
	\[
		\mass(\ell_d^{\RC}(T))=\lim_{m\to \infty}\mass(T_m)<\infty.
	\]
	Applying Theorem 4.1 in \cite{White:1999:Acta}, we get that
	$\ell_d^{\RC}(T)\in \RC_d(\mathbb{R}^n;\mathbb{Z})$.
	Let us check that $\ell_d^{\RC}(T)$ does not depend on 
	the choice of the sequence $\{T_m\}$. Indeed, if $\{T_m'\}$ is an another
	Lipschitz chains which converges to $T$ in mass, then $\mass(T-T_m')\to 0$.
	Thus
	\[
		\fn(\ell_d(T_m)-\ell_d(T_m'))\leq
		\mass(\ell_d(T_m)-\ell_d(T_m'))=\mass(T_m-T_m')\to 0,
	\]
	we get so that $\{\ell_d(T_m)\}$ and $\{\ell_d(T_m')\}$ converge to a same
	limit. Hence, the mapping $\ell_d^{\RC}:\RC_d^{loc}(\mathbb{R}^n)\to 
	\RC_d(\mathbb{R}^n;\mathbb{Z})$ is well defined.
	We claim that
	\begin{itemize}
		\item $\ell_d^{\RC}\vert_{\LC_d(\mathbb{R}^n)}=\ell_d$. For any $T\in
			\RC_d^{loc}(\mathbb{R}^n)$ and $\{T_m\}\subseteq \LC_d(\mathbb{R}^n)$,
		\item  if $T_m\ora{\fn} T$, then $ \ell_d(T_m)\ora{\fn}\ell_d^{\RC}(T)$;
		\item if $T_m\ora{\mass} T$, then $
			\ell_d(T_m)\ora{\mass}\ell_d^{\RC}(T)$.
		\item $\ell_d^{\RC}$ is an isometric isomorphism.
		\item $\ell_d^{\RC}$ commute with $\partial$ and $f_{\sharp}$ for any
			Lipschitz mapping $f:\mathbb{R}^n\to \mathbb{R}^n$.
	\end{itemize}

	The first three items easily follow from the definition of $\ell_d^{\RC}$.

	Let us go to prove that $\ell_d^{\RC}$ is an isomorphism. We first verify the
	mapping $\ell_d^{\RC}$ is an homomorphism, that is, for each $T_1, T_2\in
	\RC_d^{loc}(\mathbb{R}^n)$, we can find sequences $\{T_{1,m}\}, \{T_{2,m}\}
	\subseteq \LC_d(\mathbb{R}^n)$ such that $T_{1,m}\ora{\mass}T_1, T_{2,m}
	\ora{\mass}T_2$, thus we have that $T_{1,m}+T_{2,m}\ora{\mass}T_1+T_2$. 
	By the definition of $\ell_d^{\RC}$, we have that 
	\[
		\ell_d(T_{1,m})\ora{\fn}\ell_d^{\RC}(T_1), \
		\ell_d(T_{2,m})\ora{\fn}\ell_d^{\RC}(T_2) \text{ and }
		\ell_d(T_{1,m}+T_{2,m})\ora{\fn}\ell_d^{\RC}(T_1+T_2).
	\]
	But on the other hand, 
	\[
		\ell_d(T_{1,m}+T_{2,m})=\ell_d(T_{1,m})+
		\ell_d(T_{2,m})\ora{\fn}\ell_d^{\RC}(T_1)+\ell_d^{\RC}(T_2),
	\]
	and we get so that $\ell_d^{\RC}(T_1+T_2)=\ell_d^{\RC}(T_1)+\ell_d^{\RC}(T_2)$. 
	We next show that $\ell_d^{\RC}$ is an epimorphism. Indeed, for each $A\in
	\RC_d(\mathbb{R}^n, \mathbb{Z})$, $\varepsilon>0$, we can find Lipschitz
	chain $L\in \LC_d(\mathbb{R}^n;\mathbb{Z})$ such that $\mass(A-L)<
	\varepsilon$, thus we can find $\{L_m\}\subseteq \LC_d(\mathbb{R}^n, 
	\mathbb{Z})$ such that $L_m\ora{\mass}A$, then we have that 
	$\{\ell_d^{-1}(L_m)\}\subseteq\LC_d(\mathbb{R}^n)$ is a Cauchy sequence for 
	mass, and converges to $B\in\RC_d^{loc}(\mathbb{R}^n)$, hence $\ell_d^{\RC}(B)=A$.
	We conclude that $\ell_d^{\RC}$ is a monomorphism. Indeed, if $T_1, T_2\in
	\RC_d^{loc}(\mathbb{R}^n)$, and $\ell_d^{\RC}(T_1)=\ell_d^{\RC}(T_2)$. Then
	$\ell_d^{\RC}(T_1-T_2)=0$. But, by the definition of $\ell_d^{\RC}$, there 
	exists a sequence of Lipschitz chains $\{R_m\}\subseteq \LC_d(\mathbb{R}^n)$ 
	such that $R_m\ora{\mass} T_1-T_2$, thus we have 
	\[
		\ell_d(R_m)\ora{\mass} \ell_d^{\RC}(T_1-T_2)=0,
	\]
	thus $T_1=T_2$. 

	Now, we check that $\ell_d^{\RC}$ is an isometry. Indeed, for each $
	T\in\RC_d^{loc}(\mathbb{R}^n)$, there is a sequence of Lipschitz chains
	$T_m$ converges to $T$ in mass, we get that
	\[
		\mass(\ell_d^{\RC}(T))=\lim_{m\to \infty}\mass(\ell_d(T_m))=
		\lim_{m\to\infty}\mass(T_m)=\mass(T).
	\]

	Finally, we show that $\ell_d^{\RC}$ commute with $\partial$ and
	$f_{\sharp}$. For each $A\in\RC_d^{loc}(\mathbb{R}^n)$, there exists a sequence
	of Lipschitz chains $\{A_m\}$ such that $A_m\ora\mass A$, thus we have that
	$\ell_d(A_m)\ora{\mass} \ell_d^{\RC}(A)$, and $\partial\ell_d(A_m)\ora{\fn}
	\partial\ell_d^{\RC}(A)$. But on the other hand, we have that $\partial
	\ell_d(A_m)=\ell_{d-1}(\partial A_m)$ and $\partial A_m\ora{\fn}\partial A$,
	thus $\ell_{d-1}(\partial A_m)\ora{\fn} \ell_{d-1}^{\RC}(\partial A)$ and
	$ \partial\ell_d^{\RC}(A)=\ell_{d-1}^{\RC}(\partial A)$. Similarly, we can
	get that 
	\[
		f_{\sharp}\ell_d^{\RC}(A)=f_{\sharp}(\lim_{m\to\infty}\ell_d(A_m))=
		\lim_{m\to\infty}\ell_d(f_{\sharp}(A_m))=\ell_d^{\RC}(f_{\sharp}A).
	\]
\end{proof}
By Theorem 4.1 in \cite{White:1999:Acta}, we get that $\RC_d(\mathbb{R}^n;
G)=\{T\in \FC_{d}(\mathbb{R}^n;G):\mass(T)<\infty\}$ in case that $G$ is
discrete. Thus
$\NC_d(\mathbb{R}^n;\mathbb{Z})=\{T\in \RC_d(\mathbb{R}^n;\mathbb{Z}):
\partial T\in \RC_{d-1}(\mathbb{R}^n;\mathbb{Z})\}$. 
Since $\IC_d^{loc}(\mathbb{R}^n)=\{T\in \RC_d^{loc}(\mathbb{R}^n):\partial T
\in \RC_{d-1}^{loc}(\mathbb{R}^n)\}$, as a consequence of Lemma \ref{le:isoR},
we get the following corollary.
\begin{corollary}\label{co:isoI}
	$(\IC_d^{loc}(\mathbb{R}^n),\mathsf{N})$ and $(\NC_d(\mathbb{R}^n;\mathbb{Z}),
	\mathsf{N})$ are natural isometrically isomorphic.
\end{corollary}
For any $T\in \FC_d(\mathbb{R}^n;G)$, there exist $R\in
\FC_d(\mathbb{R}^n;G)$ and $S\in \FC_{d+1}(\mathbb{R}^n;G)$
such that $T=R+\partial S$ and $\mass(R)+\mass(S)\leq 2 \fn(T)<\infty$. Combine this with Theorem 4.1
in \cite{White:1999:Acta}, we get that 
\[
	\FC_d(\mathbb{R}^n;G)=\{R+\partial S: R\in
	\RC_d(\mathbb{R}^n;G),S\in \RC_{d+1}(\mathbb{R}^n;G)\},
\]
in case $G$ is discrete. Hence, from Lemma \ref{le:isoR}, we get the following
consequence.
\begin{corollary}\label{co:isoF}
	$(\FC_d^{loc}(\mathbb{R}^n),\fn)$ and $(\FC_d(\mathbb{R}^n;\mathbb{Z}),\fn)$ are
	natural isometric isomorphic.
\end{corollary}
\begin{proof}
	Let $\ell_{\ast}:\LC_{\ast}(\mathbb{R}^n)\to \LC_{\ast}(\mathbb{R}^n;
	\mathbb{Z})$ and $\ell_{\ast}^{\RC}:\RC_{\ast}^{loc}(\mathbb{R}^n)\to
	\RC_{\ast}(\mathbb{R}^n;\mathbb{Z})$ be the same as in Lemma \ref{le:isoR}.
	For any $T\in \FC_d^{loc}(\mathbb{R}^n)$, there exist $R\in
	\RC_d^{loc}(\mathbb{R}^n)$ and $S\in \RC_{d+1}^{loc}(\mathbb{R}^n)$ such
	that $T=R+\partial S$, we put $h_d(T)=\ell_d^{\RC}(R)+\partial
	(\ell_{d+1}^{\RC}(S))$. We will show that $h_d:\FC_d^{loc}(\mathbb{R}^n)\to
	\FC_d(\mathbb{R}^n;\mathbb{Z})$ is well defined, that is, which does not 
	depend on the decomposition $T=R+\partial S$.
	Indeed, if there exist $R_1\in\RC_d^{loc}(\mathbb{R}^n)$, 
	$S_1\in\RC_{d+1}^{loc}(\mathbb{R}^n)$ such that $T=R_1+\partial S_1$,
	then we get that $R-R_1+\partial S-\partial S_1=0$, thus
	\[
		0=\ell_d^{\RC}(R-R_1+\partial S-\partial
		S_1)=\ell_d^{\RC}(R)-\ell_d^{\RC}(R_1)+\partial
		\ell_{d+1}^{\RC}(S)-\partial \ell_{d+1}^{\RC}(S_1),
	\]
	and $\ell_d^{\RC}(R)+\partial \ell_{d+1}^{\RC}(S)=\ell_d^{\RC}(R_1)+\partial
	\ell_{d+1}^{\RC}(S_1)$.

	We now check that $h_d$ is a homeomorphism. Indeed,
	for each $T_1, T_2\in\FC_d^{loc}(\mathbb{R}^n)$, there exists $R_1,
	R_2\in\RC_d^{loc}(\mathbb{R}^n), S_1, S_2\in\RC_{d+1}^{loc}(\mathbb{R}^n)$
	such that $T_1=R_1+\partial S_1, T_2=R_2+\partial S_2$, then we have that
	\[
		\begin{aligned}
			h_d(T_1+T_2)&=\ell_d^{\RC}(R_1+R_2)+\partial(\ell_{d+1}^{\RC}(S_1+S_2))\\
			&=\ell_d^{\RC}(R_1)+\ell_d^{\RC}(R_2)+\partial(\ell_{d+1}^{\RC}(S_1))+
			\partial(\ell_{d+1}^{\RC}(S_2))\\&=h_d(T_1)+h_d(T_2).
		\end{aligned}
	\] 

	If $h_d(T)=0$ for some $T\in \FC_d^{loc}(\mathbb{R}^n)$, writing 
	$T=R+\partial S$, then we get that 
	\[
		0= h_d(T)=h_d(R+\partial S)=\ell_d^{\RC}(R)+\partial
		(\ell_{d+1}^{\RC}(S)),
	\]
	and $\partial (\ell_{d+1}^{\RC}(S))=-\ell_d^{\RC}(R)\in
	\RC_d(\mathbb{R}^n;\mathbb{Z})$, thus $\ell_{d+1}^{\RC}(S)\in
	\NC_{d+1}(\mathbb{R}^n,\mathbb{Z})$, hence $S\in
	\IC_{d+1}^{loc}(\mathbb{R}^n)$ and $\partial
	(\ell_{d+1}^{\RC}(S))=\ell_d^{\RC}(\partial S)$. Therefore,
	$\ell_d^{\RC}(R+\partial S)=\ell_d^{\RC}(R)+\ell_d^{\RC}(\partial S)=0$, and 
	$R+\partial S=0$, we get that $h_d$ is a monomorphism. 

	Let us go to prove that $h_d$ is an epimorphism. For each $A\in
	\FC_d(\mathbb{R}^n; \mathbb{Z})$, there exists $B\in\RC_d(\mathbb{R}^n;
	\mathbb{Z}), C\in\RC_{d+1}(\mathbb{R}^n, \mathbb{Z})$ such that 
	$A=B+\partial C$. We put $B_1=(\ell_d^{\RC})^{-1}(R)$ and
	$C_1=(\ell_{d+1}^{\RC})^{-1}(S)$, then $B_1\in \RC_d^{loc}(\mathbb{R}^n)$ 
	and $C_1\in\RC_{d+1}^{loc}(\mathbb{R}^n)$. We put $A_1=B_1+\partial C_1$, 
	then $A_1\in\FC_d^{loc}(\mathbb{R}^n)$ and $h_d(A_1)=A$.

	Now, we check that $h_d$ is an isometry. For any $T
	\in\FC_d^{loc}(\mathbb{R}^n)$, if $T=R+\partial S$, $R\in
	\RC_d^{loc}(\mathbb{R}^n)$ and $S\in \RC_{d+1}^{loc}(\mathbb{R}^n)$, then we
	have that $h_d(T)=\ell_d^{\RC}(R)+\partial \ell_{d+1}^{\RC}(S)$, and 
	\[
		\fn(h_d(T))\leq \mass(\ell_d^{\RC}(R))+\mass(\ell_{d+1}^{\RC}(S))=\mass(R)+\mass(S),
	\]
	thus
	\begin{equation}\label{eq:isoF50}
		\fn(h_d(T))\leq \inf\{\mass(R)+\mass(S):T=R+\partial S,R\in
		\RC_d^{loc}(\mathbb{R}^n), S\in \RC_{_d+1}^{loc}(\mathbb{R}^n)\}=\fn(T).
	\end{equation}
	Conversely, for any $T'\in \FC_d(\mathbb{R}^n;\mathbb{Z})$, if
	$T'=R'+\partial S'$, $R'\in \RC_d(\mathbb{R}^n;\mathbb{Z})$ and $S'\in
	\RC_{d+1}(\mathbb{R}^n;\mathbb{Z})$, then we see that 
	$R=(\ell_d^{\RC})^{-1}(R')\in \RC_d^{loc}(\mathbb{R}^n)$ and 
	$S=(\ell_{d+1}^{\RC})^{-1}(S')\in \RC_{d+1}^{loc}(\mathbb{R}^n)$, and that 
	$h_d(R+\partial S)=T'$. We get that 
	\[
		\fn(h_d^{-1}(T'))=\fn(R+\partial S)\leq
		\mass(R)+\mass(S)=\mass(R')+\mass(S'),
	\]
	thus
	\begin{equation}\label{eq:isoF60}
		\fn(h_d^{-1}(T'))\leq \fn(T').
	\end{equation}
	From \eqref{eq:isoF50} and \eqref{eq:isoF60}, we get that $h_d$ is an
	isometry.

	The commutativity of $h_d$ and $\partial$ follows from the definition of
	$h_d$.

	Finally, we show that $h_d$ commute with $f_{\sharp}$.
	Indeed, for each $A\in\FC_d^{loc}(\mathbb{R}^n)$, writing $A=B+\partial C$,
	$B\in\RC_d^{loc}(\mathbb{R}^n)$, $C\in\RC_{d+1}^{loc}(\mathbb{R}^n)$, then
	$h_d(A)=\ell_d^{\RC}(B)+\partial (\ell_{d+1}^{\RC}(C))$, and 
	\[
		f_{\sharp}h_d(A)=f_{\sharp}\ell_d^{\RC}(B)+f_{\sharp}\partial \ell_{d+1}^{\RC}(C))=\ell_d^{\RC}(f_{\sharp}(B))+\partial\ell_{d+1}^{\RC}(f_{\sharp}(C))=h_d(f_{\sharp}A).	
	\]
\end{proof}

\begin{proposition}
	\label{prop:intsize}
	Let $B_0\subseteq \mathbb{R}^{n}$ be a compact Lipschitz neighborhood retract
	with $\HM^d(B_0)=0$. Then for any $\mathscr{S}\in \{\IC,\RC\}$
	and $T\in \mathscr{S}_{d-1}(B_0)$ with $\partial T=0$, we have that 
	\[
		\inf\{\size(S):S\in\mathscr{S}_d(\mathbb{R}^n), \partial S=T\}=\inf\{
		\HM^d(E\setminus B_0):E\in \mathscr{C}_{\mathscr{S}}(B_0,\mathbb{Z},[T])\}.
	\]
\end{proposition}
\begin{proof}
	The proof easily follows from Proposition \ref{prop:flatsize}, Lemma
	\ref{le:isoR}, Corollary \ref{co:isoI}, Corollary \ref{co:isoF} and the
	coincidence of the Hausdorff size and the flat size for rectifiable flat
	chains.
\end{proof}

\section{Comparison of homology}
\begin{proposition}\label{prop:homeq}
	Let $G$ be an abelian group. Let $X\subseteq\mathbb{R}^{n}$ be a compact subset
	which is a neighborhood retract. Suppose $\rho:U\to X$ is a retraction.  Let
	$\Hom_{\ast}$ and $\Hom_{\ast}'$ be any two homologies satisfying the 
	axioms 1, 2 and 4 of Eilenberg and Steenrod. If there is a set $V$ which is a
	union of finite number of polyhedra and $X\subseteq V\subseteq U$, and there
	is an isomorphism $h_V:\Hom_{\ast}(V;G)\to \Hom_{\ast}'(V;G)$ such that
	$h_V\circ \Hom_{\ast}(i_{X,V}\circ \rho\vert_V)=\Hom_{\ast}'(i_{X,V}\circ
	\rho\vert_V)\circ h_V$, then $h_V(\ker \Hom_{\ast}
	(\rho\vert_V))=\ker\Hom_{\ast}'(\rho\vert_V)$ and 
	\[
		\Hom_{\ast}(X;G)\cong\Hom_{\ast}'(X;G).
	\]
\end{proposition}
\begin{proof}
	Since $\rho\vert_X=\id_X$, we have that $\rho\vert_V \circ i_{X,V}=\id_X$,
	thus 
	\[
		\Hom_{\ast}(\rho\vert_V) \circ \Hom_{\ast}(i_{X,V})=\id_{\Hom_{\ast}(X;G)}
		\text{ and }
		\Hom_{\ast}'(\rho\vert_V) \circ
		\Hom_{\ast}'(i_{X,V})=\id_{\Hom_{\ast}'(X;G)},
	\]
	hence $\Hom_{\ast}(i_{X,V})$ and $\Hom_{\ast}'(i_{X,V})$ are monomorphism,
	$\Hom_{\ast}(\rho\vert_V)$ and $\Hom_{\ast}'(\rho\vert_V)$ are epimorphism.
	Since $\Hom_{\ast}(i_{X,V}\circ \rho\vert_V)=\Hom_{\ast}(i_{X,V})\circ
	\Hom_{\ast}(\rho\vert_V)$, $\Hom_{\ast}'(i_{X,V}\circ
	\rho\vert_V)=\Hom_{\ast}'(i_{X,V})\circ \Hom_{\ast}'(\rho\vert_V)$, and 
	$h_V\circ \Hom_{\ast}(i_{X,V}\circ \rho\vert_V)=\Hom_{\ast}'(i_{X,V}\circ
	\rho\vert_V)\circ h_V$, we have that 
	\[
		h_V(\ker \Hom_{\ast} (\rho\vert_V))=\ker\Hom_{\ast}'(\rho\vert_V).
	\]

	We now consider the homomorphism
	\[
		h_X=\Hom_{\ast}'(\rho\vert_V)\circ h_V \circ
		\Hom_{\ast}(i_{X,V}):\Hom_{\ast}(X;G)\to \Hom_{\ast}'(X;G).
	\]
	Since $h_V\circ \Hom_{\ast}(i_{X,V}\circ \rho\vert_V)=\Hom_{\ast}'(i_{X,V}\circ
	\rho\vert_V)\circ h_V$, we get that 
	\[
		h_X\circ \Hom_{\ast}(\rho\vert_V)=\Hom_{\ast}'(\rho\vert_V)\circ h_V\circ
		\Hom_{\ast}(i_{X,V})\circ
		\Hom_{\ast}(\rho\vert_V)=\Hom_{\ast}'(\rho\vert_V)\circ h_V.
	\]
	But $h_V$ is an isomorphism and $\Hom_{\ast}'(\rho\vert_V)$ is an
	epimorphism, we get that $h_X$ is also an epimorphism. Similarly, we have
	that 
	\[
		\Hom_{\ast}'(i_{X,V})\circ h_V=h_V\circ \Hom_{\ast}(i_{X,V}).
	\]
	Since $h_V$ is an isomorphism, $\Hom_{\ast}'(i_{X,V})$ and
	$\Hom_{\ast}(i_{X,V})$ are monomorphism, we get that $h_X$ is a
	monomorphism.

\end{proof}
\begin{corollary} \label{co:sc}
	Let $G$ be an abelian group. Let $X\subseteq\mathbb{R}^{n}$ be a compact subset
	which is a neighborhood retract. Let $\Hom_{\ast}^s$ be the singular homology.
	Then the natural homomorphism  
	\[
		\Hom_{\ast}^s(X;G)\to\cech_{\ast}(X;G)
	\]
	is an isomorphism.
\end{corollary}
\begin{proof}
	It is clear there is a natural homomorphism
	$v_{\ast}:\Hom_{\ast}^s(\cdot;G)\to \cech_{\ast}(\cdot;G)$, which is commute
	with any induced homomorphism $f_{\ast}$, see for example
	\cite{Mardesic:1959}. We assume that $U\supseteq X$ is an open set,
	$\rho:U\to X$ is a retraction, and take $V$ a closed set which is a union of
	finite number of polyhedra, and satisfies that $X\subseteq V\subseteq U$. 
	Then we have that the natural homomorphism $v_{\ast}^{V}:\Hom_{\ast}^s(V;G)\to
	\cech_{\ast}(V;G)$ is an isomorphism, and $v_{\ast}^{V}\circ
	\Hom_{\ast}^s(i_{X,V}\circ \rho\vert_V)=\cech_{\ast}(i_{X,V}\circ
	\rho\vert_V)\circ v_{\ast}^{V}$.
	Thus $v_{\ast}^{X}=\cech_{\ast}(\rho\vert_V)\circ v_{\ast}^{V}\circ
	\Hom_{\ast}^s(i_{X,V})$, and by Proposition \ref{prop:homeq}, we get that
	$v_{\ast}^{X}:\Hom_{\ast}^s(X;G)\to \cech_{\ast}(X;G)$ is an isomorphism.
\end{proof}
\begin{proposition}\label{prop:hfc}
	Let $X\subseteq\mathbb{R}^{n}$ be a compact Lipschitz neighborhood retract.
	Then for any $\mathscr{S}\in \{\LC,\NC,\RC,\FC\}$ and
	complete normed abelian group $G$, we have that
	\[
		\Hom_{\ast}^{\mathscr{S}}(X;G)\cong \cech_{\ast}(X;G).
	\]
\end{proposition}
\begin{proof}
	For any set $Y\subseteq \mathbb{R}^n$ and $\mathscr{S}\in \{\NC,\RC,\FC\}$, 
	we see that $\LC_{\ast}(Y;G)\subseteq \mathscr{S}_{\ast}(Y,G)$, thus 
	there is a natural homomorphism $h_{Y}=h_{Y}^{\mathscr{S}}:\Hom_{\ast}^{\LC}
	(Y;G)\to\Hom_{\ast}^{\mathscr{S}}(Y;G)$ defined by
	\[
		h_Y(T+\mathcal{B}_{\ast}^{\LC}(Y;G))=T+\mathcal{B}_{\ast}^{\FC}(Y;G),\
		\forall T\in \mathcal{Z}_{\ast}^{\LC}(Y;G).
	\]
	We will show that $h_X$ is an isomorphism. Indeed, for any $T_1,T_2\in
	\mathcal{Z}_{\ast}^{\LC}(X;G)$, if $T_1-T_2\not\in
	\mathcal{B}_{\ast}^{\LC}(X;G)$, then $T_1-T_2\not\in
	\mathcal{B}_{\ast}^{\FC}(X;G)$, thus $h_X$ is a monomorphism. 
	For any $T\in \mathcal{Z}_{\ast}^{\mathscr{S}}(X;G)$, by Lemma
	\ref{le:appl}, we can find a sequence of Lipschitz chains $\{T_m\}\subseteq
	\LC_{\ast}(X;G)$ such that $T_m\ora{\fn} T$ and $\partial T_m=0$. By Lemma
	\ref{le:fnsm}, we can find $R_m\in \mathscr{S}_{\ast+1}(X;G)$ such that
	$T-T_m=\partial R_m$. Thus $T+\mathcal{B}_{\ast}^{\mathscr{S}}(X;G)=
	T_m+\mathcal{B}_{\ast}^{\mathscr{S}}(X;G)$, and 
	\[
		h_X(T_m+\mathcal{B}_{\ast}^{\LC}(X;G))=
		T+\mathcal{B}_{\ast}^{\mathscr{S}}(X;G).
	\]
	Hence, $h_X$ is an epimorphism, and we get that $h_X$ is an isomorphism.

	For any $Y\subseteq \mathbb{R}^n$, we denote by
	$C_{\ast}(Y;G)$ the group of singular chains, then $\LC_{\ast}(X;G)\subseteq
	C_{\ast}(X;G)$, so there is a natural 
	homomorphism $\varphi_V:\Hom_{\ast}^{\LC}(V;G)\to\Hom_{\ast}^{s}(V;G)$.
	Since $X$ is a Lipschitz neighborhood retract, we can find open set
	$U\supseteq X$ and Lipschitz mapping $\rho:U\to X$ such that
	$\rho\vert_X=\id_X$. We take a closed set $V$ which is a union of finite
	number of polyhedra and satisfies that $X\subseteq V\subseteq U$. Then,
	by Proposition 1.3 in \cite{Yamaguchi:1997}, we see that the natural 
	homomorphism $\varphi_V:\Hom_{\ast}^{\LC}(V;G)\to\Hom_{\ast}^{s}(V;G)$ is 
	an isomorphism. It is easy to see that, for any Lipschitz mapping 
	$f:X\to Z$, we have that $\varphi_Z\circ
	\Hom_{\ast}^{\LC}(f)=\Hom_{\ast}^{s}(f)\circ \varphi_X$. By an analogous
	argument as in the proof of Corollary \ref{co:sc}, we may get that
	$\varphi_X$ is an isomorphism. Put $\Psi_{\ast}^{\mathscr{S}}=v_{\ast}^{X}
	\circ \varphi_X\circ (h_{X}^{\mathscr{S}})^{-1}$. Then
	$\Psi_{\ast}^{\mathscr{S}}:\Hom_{\ast}^{\mathscr{S}}(X;G)\to
	\cech_{\ast}(X;G)$ is an isomorphism.
\end{proof}

\section{Infimum of Plateau's problem}
\begin{proof}[Proof of Theorem \ref{thm:hequ}]
	We assume $F\leq b<+\infty$ and $\lambda\leq b<+\infty$.
	Let $U_0$ be an open set which contains $B_0$, and $\rho:U_0\to B_0$ be a 
	Lipschitz retraction. We take $\varepsilon>0$ such that $B_0+\oball(0,10
	\varepsilon)\subseteq U_0$, and take $B_{\varepsilon}$ to be a union of 
	a finite number of polyhedra such that 
	\[
		B_0+\oball(0,2 \varepsilon)\subseteq
		B_{\varepsilon}\subseteq B_0+\oball(0,3 \varepsilon).
	\]

	For any $E\in \ch(B_0,G,L)$, we choose $R_0$ large enough such that
	$E\subseteq \cball(0,R_0-1)$. Applying Theorem
	\ref{thm:polyapp} with $\varepsilon_1=\varepsilon_2=\varepsilon$,
	$U=\oball(0,R_0)\setminus B_0$, $X=E$,
	$K=\cball(0,R_0-1)\setminus (B_0+\oball(0,\varepsilon)$,
	we get a Lipschitz mapping $\varphi_{\varepsilon} :\mathbb{R}^n\to 
	\mathbb{R}^n$ such that \ref{polyapp1}, \ref{polyapp2}, \ref{polyapp3} and
	\ref{polyapp4} hold. We put $E_{\varepsilon}=\varphi_{\varepsilon}(E)$, and 
	$\mathcal{O}_{\varepsilon}=\cball(0,R_0)\setminus (B_0+\oball(0,4
	\varepsilon))$. Let $\bar{\psi}_{\varepsilon}:\mathbb{R}^n\to \mathbb{R}^n$
	be a Lipschitz extension of $\psi_{\varepsilon}$ with  
	$ \Lip(\bar{\psi}_{\varepsilon})= \Lip(\psi_{\varepsilon})$, where
	$\psi_{\varepsilon}:\mathcal{O}_{\varepsilon}\cup B_{\varepsilon}\to
	\mathbb{R}^n$ is defined by
	\[
		\psi_{\varepsilon}(x)=\begin{cases}
			x,&x\in \mathcal{O}_{\varepsilon},\\
			\rho(x),&x\in B_{\varepsilon}.
		\end{cases}
	\]
	For each $x\in \mathcal{O}_{\varepsilon}$, $y\in  B_{\varepsilon}$, we choose
	$y_0\in B_0$, such that$ \left|y-y_0 \right|=\dist(y,B_0)\leq 3\varepsilon $,
	then we have
	\[
		\begin{aligned}
			\left|\psi_{\varepsilon}(x)-\psi_{\varepsilon}(y)\right|&=	
			\left|x-\rho(y)\right| \leq \left|x-y \right|+\left|y-y_0 \right|+
			\left|y_0-\rho(y)\right| \\
			&\leq \left|x-y \right|+3\varepsilon+3\Lip(\rho)\varepsilon 
			\leq \left|x-y \right|+3(1+\Lip(\rho)) \left|x-y \right| \\
			&\leq (4+3\Lip(\rho)) \left|x-y \right|.
		\end{aligned}	
	\]
	Thus $\Lip(\bar{\psi}_{\varepsilon})=\Lip(\psi_{\varepsilon})\leq 4+3\Lip(\rho))$.

	We will show that $B_{\varepsilon}\cup E_{\varepsilon}$ spans $L$ in Homology
	$\Hom_{\ast}$. Indeed, we define the inclusion mappings
	\[
		i_1:B_0\hookrightarrow E,\ i_2:E_{\varepsilon}	\hookrightarrow
		B_{\varepsilon}\cup E_{\varepsilon}.
	\]
	Since $\varphi_{\varepsilon}$ is homotopic to identity, we can get
	$ i_2\circ\varphi_{\varepsilon}\circ i_1$ is homotopic to $ i_{B_0,
	B_{\varepsilon}\cup E_{\varepsilon}}$, thus
	\[
		\Hom_{\ast}(i_2\circ\varphi_{\varepsilon}\circ i_1)=
		\Hom_{\ast}(i_{B_0, B_{\varepsilon}\cup E_{\varepsilon}}),
	\]
	and we get that 
	\[
		B_{\varepsilon}\cup E_{\varepsilon}\in  \ch(B_0,G,L).
	\]
	Since $B_{\varepsilon}\cup E_{\varepsilon}$ is a union of finite number of
	polyhedra, we have that 
	\[
		B_{\varepsilon}\cup E_{\varepsilon}\in  \cc(B_0,G,L).
	\]
	Since $\bar{\psi}_{\varepsilon}|_{B_0}=\rho|_{B_0}=\id_{B_0}$, we have that
	$\bar{\psi}_{\varepsilon}\circ  i_{B_0,  B_{\varepsilon}\cup
	E_{\varepsilon}}= i_{B_0, \bar{\psi}_{\varepsilon}(B_{\varepsilon}\cup
	E_{\varepsilon})}$, and 
	\[
		\Hom_{\ast}(\bar{\psi}_{\varepsilon}\circ  i_{B_0,  B_{\varepsilon}\cup
		E_{\varepsilon}})=\Hom_{\ast}(i_{B_0,
		\bar{\psi}_{\varepsilon}(B_{\varepsilon}\cup E_{\varepsilon})}),
	\]
	thus
	\[
		\bar{\psi}_{\varepsilon}(B_{\varepsilon}\cup E_{\varepsilon})\in  \cc(B_0,G,L).
	\]
	Since $ B_{\varepsilon}\cup E_{\varepsilon}=(E_{\varepsilon}\cap
	\mathcal{O}_{\varepsilon})\cup((B_{\varepsilon}\cup
	E_{\varepsilon})\setminus\mathcal{O}_{\varepsilon})$ and 
	\[
		\bar{\psi}_{\varepsilon}((B_{\varepsilon}\cup E_{\varepsilon})\setminus
		\mathcal{O}_{\varepsilon})=\bar{\psi}_{\varepsilon}(B_{\varepsilon})
		\cup\bar{\psi}_{\varepsilon}( E_{\varepsilon}\setminus
		\mathcal{O}_{\varepsilon}) =B_0\cup \bar{\psi}_{\varepsilon}
		(E_{\varepsilon}\setminus\mathcal{O}_{\varepsilon}) =\bar{\psi}_{
		\varepsilon}(E_{\varepsilon}\setminus\mathcal{O}_{\varepsilon}),
	\]
	we have that 
	\[
		\begin{aligned}
			\Phi_{F,\lambda}((\bar{\psi}_{\varepsilon}(B_{\varepsilon}\cup
			E_{\varepsilon}))\setminus B_0)&\leq \Phi_{F,\lambda}
			(\bar{\psi}_{\varepsilon}(E_{\varepsilon}\cap\mathcal{O}_{\varepsilon}))
			+\Phi_{F,\lambda}(\bar{\psi}_{\varepsilon}((B_{\varepsilon}\cup
			E_{\varepsilon})\setminus\mathcal{O}_{\varepsilon})\setminus B_0) \\
			&\leq \Phi_{F,\lambda}(E_{\varepsilon}\cap \mathcal{O}_{\varepsilon})+
			\Phi_{F,\lambda}(\bar{\psi}_{\varepsilon}(E_{\varepsilon}\setminus
			\mathcal{O}_{\varepsilon})\setminus B_0) \\
			&\leq \Phi_{F,\lambda}(\varphi_{\varepsilon}(E\setminus B_0))+
			b \Lip(\bar{\psi}_{\varepsilon})^d\HM^d((E_{\varepsilon}\setminus
			\mathcal{O}_{\varepsilon})\setminus B_0) \\
			&\leq \Phi_{F,\lambda}(E\setminus B_0)+b\Lip(\bar{\psi}_{\varepsilon} )^d
			\HM^d(\varphi_{\varepsilon}(E\cap (B_0+\oball(0,5
			\varepsilon)))\setminus B_0) +\varepsilon \\
			&\leq \Phi_{F,\lambda} (E\setminus B_0)+bC\Lip(\psi_{\varepsilon})^d
			\HM^d((E\cap(B_0+\oball(0,5\varepsilon)))\setminus B_0)+\varepsilon \\
			&\leq \Phi_{F,\lambda} (E\setminus
			B_0)+C_2\HM^d((E\cap(B_0+\oball(0,5\varepsilon)))\setminus
			B_0)+\varepsilon,
		\end{aligned}
	\]
	where $C_2=bC\cdot(4+3\Lip(\rho))^d$ and $C$ is the constant in Theorem
	\ref{thm:polyapp}. But we see that $\HM^d(E\cap (B_0+\oball(0,5 \varepsilon))
	\setminus B_0)\to 0$ as $\varepsilon\to 0$, we get so that 
	\[
		\inf\{\Phi_{F,\lambda}(E\setminus B_0):E\in \cc(B_0,G,L)\}\leq 
		\inf\{\Phi_{F,\lambda}(E\setminus B_0):E\in \ch(B_0,G,L)\}.
	\]
	By a similar argument as above, we will obtain the reverse inequality 
	\[
		\inf\{\Phi_{F,\lambda}(E\setminus B_0):E\in \ch(B_0,G,L)\}\leq 
		\inf\{\Phi_{F,\lambda}(E\setminus B_0):E\in \cc(B_0,G,L)\}.
	\]

\end{proof}

\begin{proof}[Proof of Theorem \ref{thm:scequ}]
	By Lemma \ref{le:cnon}, we only need to consider the case $\mathscr{S}\in
	\{\LC,\NC,\RC,\FC\}$. By Proposition \ref{prop:flatsize}, we get that 
	\[
		\inf\{\size(S): \partial S=T,S\in\mathscr{S}_{d}(\mathbb{R}^{n};G)\}
		=\inf\{\HM^{d}(E): E\in \mathscr{C}_{\mathscr{S}}(B_0,G,T)\}.
	\]
	By Theorem \ref{thm:hequ} and Proposition \ref{prop:hfc}, we get that 
	\[
		\inf\{\HM^{d}(E): E\in \mathscr{C}_{\mathscr{S}}(B_0,G,T)\}=
		\inf\{\HM^{d}(E): E\in \cc(B_0,G,T)\}.
	\]
	Thus \eqref{eq:scequ} holds.
\end{proof}
\begin{proof}[Proof of Corollary \ref{co:scequ}]
	It follows from Theorem \ref{thm:scequ} and Lemma
	\ref{le:isoR} and Corollary \ref{co:isoI}.	
\end{proof}
\begin{corollary}
	Let $B_0\subseteq \mathbb{R}^{n}$ be a compact $C^{1,\alpha}$ submanifold of
	dimension $d-1$ without boundary, $0<\alpha\leq 1$. Suppose that $L$ is a 
	subgroup of $\cech_{d-1}(B_0;\mathbb{Z})$, $E$ is a \v Cech minimizer. If
	$\Theta_E^{d}(x)=1/2$ for any $x\in B_0$,
	then $E$ is a minimizer with the singular homology boundary conditions,
	and 
	\begin{equation}\label{eq:half}
		\HM^d(E)=\inf\{\size(T):T\in \IC_{d}(\mathbb{R}^n), [\partial T]\in L\}.
	\end{equation}
\end{corollary}
\begin{proof}
	Put $U=\mathbb{R}^n\setminus B_0$. Then $E\cap U$ is minimal in $U$. By
	Theorem 16.1 in \cite{David:2009}, for any $x\in E\cap U$, there is a ball
	$\oball(x,r)$ such that $E\cap \oball(x,r)$ is locally biH\"older equivalent 
	to a minimal cone in $\mathbb{R}^n$ at $x$. By the Allard-type boundary 
	regularity theorem \cite{Allard:1975, Bourni:2016}, we get that for any
	$x\in B_0$, there is a ball $\oball(x,r)$ such that $E\cap \oball(x,r)$ is
	$C^{1,\gamma}$ equivalent to a half $d$-plane for some $0<\gamma<\alpha$. 
	Thus $E$ is a local H\"older
	neighborhood retract. That is, there is an open set $U$ with $E\subseteq U$,
	and H\"older mapping $\rho:U\to E$ such that $\rho\vert_E=\id_E$. Let $E_m$
	be a minimizing sequence under the singular homology conditions such that 
	$E_m$ converges to $E$ in Hausdorff distance, then $E_m\subseteq U$ for $m$ 
	large enough, thus $\rho(E_m)\subseteq E$, and $\rho(E_m)$ are minimizers.
	Equality \eqref{eq:half} is easily follows from Corollary \ref{co:scequ}.
\end{proof}
\begin{proof}[Proof of Proposition \ref{prop:com}]
	Put $U=\mathbb{R}^3\setminus M$.
%	For any $T\in \IC_2(\mathbb{R}^3)$ with
%	$\spt \partial T\subseteq B_0$, we have that $\spt \partial (T\mr U)\subseteq
%	B_0$ and $\partial T-\partial (T\mr U)=\partial S$ for some $S\in
%	\IC_2(B_0)$, thus 
%	\[
%		\inf\{\size(T):[T]\in \Hom_2^{\IC}(\mathbb{R}^3;M),[\partial
%		T]=\sigma\}=\inf\{\size(T\mr U):[T]\in \Hom_2^{\IC}(\mathbb{R}^n;M),[\partial
%		T]=\sigma\},
%	\]
%	so 
	By Proposition \ref{prop:fssb},
	\[
		\inf\{\size(T):[T]\in \Hom_2^{\IC}(\mathbb{R}^3;M),[\partial
		T]=\sigma\}=\inf\{\HM^2(F\setminus M):F\in \cc(M,\mathbb{Z},\sigma)\}.
	\]
	We take a sequence of compact sets $\{E_m\}\subseteq \cc(M,\mathbb{Z},\sigma)$ 
	such that $E_m\ora{\HD}E\in \cc(M,\mathbb{Z},\sigma)$, 
	\[
		\HM^2(E\setminus M)=\lim_{m\to \infty}\HM^2(E_m\setminus M)=
		\inf\{\HM^2(F\setminus M):F\in \cc(M,\mathbb{Z},\sigma)\},
	\]
	and there exists $T_m\in \IC_2(\mathbb{R}^3)$ satisfying $T_m\mr E_m=T_m$ and
	$[\partial T_m]=\sigma$. $E$ is a \v Cech minimizer, thus $E\cap U$ is minimal
	in $U$. By Jean Taylor's regularity theorem \cite{Taylor:1976}, we get that at 
	every point $x\in E\cap U$, $E\cap U$ is locally $C^{1,\gamma}$ diffeomorphic 
	to a minimal cone in $\mathbb{R}^3$. By Theorem 1.2 in \cite{Fang:2021}, we 
	get that at every point $x\in E\cap M$, $E$ is locally $C^{1,\gamma}$ 
	diffeomorphic to a minimal cone in the half space. Hence, $E$ is a
	Lipschitz neighborhood retract, and there exit minimizers which are contained
	in $E$.  
\end{proof}

\bigskip
\footnotesize

\textsc{Yangqin FANG, School of Mathematics and Statistics,
Huazhong University of Science and Technology,
430074, Wuhan, P.R. China}\par\nopagebreak
\textit{E-mail address}: \texttt{yangqinfang@hust.edu.cn}

\medskip

\textsc{Vincent FEUVRIER, Institut de Math\'ematiques de Toulouse,
Universit\'e Paul Sabatier de Toulouse, 118 Route de Narbonne
31400, Toulouse, France}\par\nopagebreak
\textit{E-mail address}: \texttt{vincent.feuvrier@math.univ-toulouse.fr}

\medskip

\textsc{Chunyan LIU, School of Mathematics and Statistics,
Huazhong University of Science and Technology,
430074, Wuhan, P.R. China}\par\nopagebreak
\textit{E-mail address}: \texttt{chunyanliu@hust.edu.cn}


\begin{bibdiv}
\begin{biblist}

\bib{Allard:1975}{article}{
      author={Allard, W.},
       title={On the first variation of a varifold: boundary behavior},
        date={1975},
     journal={Ann. of Math.},
       pages={418\ndash 446},
}

\bib{Almgren:1968}{article}{
      author={Almgren, F.},
       title={Existence and regularity almost everywhere of solutions to
  elliptic variational problems among surfaces of varying topological type and
  singularity structure},
        date={1968},
     journal={Ann. of Math.},
      volume={87},
      number={2},
       pages={321\ndash 391},
}

\bib{Bourni:2016}{article}{
      author={Bourni, Theodora},
       title={Allard-type boundary regularity for ${C}^{1,\alpha}$ boundaries},
        date={2016},
     journal={Advances in Calculus of Variations},
      volume={9},
      number={2},
       pages={143\ndash 161},
}

\bib{David:2009}{article}{
      author={David, G.},
       title={H\"older regularity of two-dimensional almost-minimal sets in
  {$\mathbb{R}^n$}},
        date={2009},
     journal={Annales de la faculté des sciences de Toulouse},
      volume={18},
      number={1},
       pages={65\ndash 246},
}

\bib{David:2012}{article}{
      author={David, G.},
       title={Should we solve {Plateau's} problem again?},
        date={2012},
     journal={Advances in Analysis: The Legacy of Elias M. Stein. Edited by C.
  Fefferman, A. D. Ionescu, D. H. Phong, and S. Wainger, Princeton Mathematical
  Series},
      volume={50},
       pages={108\ndash 145},
}

\bib{Douglas:1931}{article}{
      author={Douglas, Jesse},
       title={Solutions of the problem of {Plateau}},
        date={1931},
     journal={Trans. Amer. Math. Soc.},
      volume={33},
       pages={263\ndash 321},
}

\bib{ES:1945}{article}{
      author={Eilenberg, S.},
      author={Steenrod, N.},
       title={Axiomatic approach to homology theory},
        date={1945},
        ISSN={0027-8424},
     journal={Proc. Nat. Acad. Sci. U. S. A.},
      volume={31},
       pages={117\ndash 120},
}

\bib{ES:1952}{book}{
      author={Eilenberg, S.},
      author={Steenrod, N.},
       title={Foundations of algebraic topology},
   publisher={Princeton},
        date={1952},
}

\bib{Fang:2013}{article}{
      author={Fang, Yangqin},
       title={Existence of minimizers for the {Reifenberg Plateau} problem},
        date={2016},
     journal={Ann. Scuola Norm. Sup. Pisa Cl. Sci.},
      volume={16},
      number={5},
       pages={817\ndash 844},
}

\bib{Fang:2021}{article}{
      author={Fang, Yangqin},
       title={Local ${C}^{1,\beta}$-regularity at the boundary of two
  dimensional sliding almost minimal sets in $\mathbb{R}^3$},
        date={2021},
     journal={Trans. Amer. Math. Soc. Ser. B},
      volume={8},
       pages={130\ndash 189},
}

\bib{FS:2018}{article}{
      author={Fang, Yangqin},
      author={Kolasi{\'{n}}ski, S{\l}awomir},
       title={Existence of solutions to a general geometric elliptic
  variational problem},
        date={2018},
     journal={Calculus of Variations and Partial Differential Equations
  (accepted)},
}

\bib{Federer:1969}{book}{
      author={Federer, H.},
       title={Geometric measure theory},
   publisher={Springer-Verlag, New York},
        date={1969},
}

\bib{FF:1960}{article}{
      author={Federer, H.},
      author={Fleming, W.},
       title={Normal and integral currents},
        date={1960},
     journal={Ann. of Math.},
      volume={72},
      number={3},
       pages={458\ndash 520},
}

\bib{Feuvrier:2008}{thesis}{
      author={Feuvrier, V.},
       title={Un r\'esultat d'existence pour les ensembles minimaux par
  optimisation sur des grilles poly\'edrales},
        type={Ph.D. Thesis},
        date={2008},
}

\bib{Feuvrier:2012}{article}{
      author={Feuvrier, V.},
       title={Remplissage de l'espace euclidien par des complexes
  poly\'edriques d'orientation impos\'ee et de rotondit\'e uniforme},
        date={2012},
     journal={Bulletin de la Soci\'et\'e Math\'ematique de France},
      volume={140},
      number={2},
       pages={163\ndash 235},
         url={http://www.numdam.org/item/BSMF_2012__140_2_163_0},
}

\bib{Fleming:1966}{article}{
      author={Fleming, W.},
       title={Flat chains over a finite coefficient group},
        date={1966},
     journal={Trans. Amer. Math. Soc.},
      volume={121},
       pages={160\ndash 186},
}

\bib{Mardesic:1959}{article}{
      author={Marde\v{s}i\'{c}, Sibe},
       title={Comparison of singular and \v{C}ech homology in locally connected
  spaces.},
        date={1959},
     journal={Michigan Math. J.},
      volume={6},
       pages={151\ndash 166},
}

\bib{Mattila:1995}{book}{
      author={Mattila, P.},
       title={Geometry of sets and measures in euclidean spaces: Fractals and
  rectifiability},
      series={Cambridge Studies in Advanced Mathematics},
   publisher={Cambridge U. Press},
        date={1995},
      volume={44},
}

\bib{Milnor:1962}{article}{
      author={Milnor, J.},
       title={On axiomatic homology theory},
        date={1962},
     journal={Pacific J. Math.},
      volume={12},
      number={1},
       pages={337\ndash 341},
}

\bib{Morgan:1989}{article}{
      author={Morgan, F.},
       title={Size-minimizing rectifiable currents},
        date={1989},
     journal={Invent. Math.},
      volume={96},
      number={2},
       pages={333\ndash 348},
}

\bib{Osserman:1970}{article}{
      author={Osserman, R.},
       title={A proof of the regularity everywhere of the classical solution to
  {Plateau's} problem},
        date={1970},
        ISSN={0003486X},
     journal={Ann. of Math.},
      volume={91},
      number={3},
       pages={550\ndash 569},
         url={http://www.jstor.org/stable/1970637},
}

\bib{Pauw:2007}{article}{
      author={Pauw, T.~De},
       title={Comparing homologies: \u{C}ech's theory, singular chains,
  integral flat chains and integral currents},
        date={2007},
     journal={Rev. Mat. Iberoamericana},
      volume={23},
}

\bib{Pauw:2009}{article}{
      author={Pauw, T.~De},
       title={Size minimizing surfaces},
        date={2009},
     journal={Ann. Sci. Ecole Norm. Sup.},
      volume={42},
      number={1},
       pages={37\ndash 101},
}

\bib{Rado:1930}{article}{
      author={Rad{\'o}, T.},
       title={The problem of the least area and the problem of {Plateau}},
        date={1930},
     journal={Math. Z.},
      volume={32},
      number={1},
       pages={763\ndash 796},
}

\bib{Reifenberg:1960}{article}{
      author={Reifenberg, E.~R.},
       title={{Solution of the Plateau problem for $m$-dimensional surfaces of
  varying topological type}},
        date={1960},
     journal={Acta Mathematica},
      volume={104},
       pages={1 \ndash  92},
         url={https://doi.org/10.1007/BF02547186},
}

\bib{Taylor:1976}{article}{
      author={Taylor, J.~E.},
       title={{The structure of singularities in soap-bubble-like and
  soap-film-like minimal surfaces}},
        date={1976},
     journal={Ann. of Math},
      volume={103},
       pages={489\ndash 539},
}

\bib{White:1999:Acta}{article}{
      author={White, B.},
       title={The deformation theorem for flat chains},
        date={1999},
     journal={Acta Math.},
      volume={183},
       pages={255\ndash 271},
}

\bib{White:1999:Ann}{article}{
      author={White, B.},
       title={Rectifiability of flat chains},
        date={1999},
     journal={Ann. of Math.},
      volume={150},
      number={1},
       pages={165\ndash 184},
}

\bib{Yamaguchi:1997}{article}{
      author={Yamaguchi, Takao},
       title={Simplicial volumes of alexandrov spaces},
        date={1997},
     journal={Kyushu Journal of Mathematics},
      volume={51},
      number={2},
       pages={273\ndash 296},
}

\end{biblist}
\end{bibdiv}
\end{document}